\pdfoutput=1
\documentclass[runningheads]{llncs}
\usepackage[T1]{fontenc}
\usepackage{graphicx}
\usepackage{float}
\usepackage{amsmath}
\usepackage{mathtools}
\usepackage{amssymb}
\usepackage{subfig}
\usepackage{tikz}
\usepackage[hidelinks]{hyperref}
\usepackage{microtype}
\usepackage[symbol,perpage]{footmisc}
\usepackage{array}

\usetikzlibrary{matrix}
\definecolor{light}{HTML}{C6DEBD}
\definecolor{dark}{HTML}{218439}
\newcommand{\R}{\text{R}}
\renewcommand{\S}{\text{S}}

\newcommand{\U}{\text{U}}
\newcommand{\V}{\text{V}}
\newcommand{\W}{\text{W}}
\newcommand{\X}{\text{X}}
\newcommand{\Z}{\mathbb{Z}}
\newcommand{\w}{\texttt{w}}
\renewcommand{\d}{\texttt{d}}

\newcommand{\limp}{\mathbin\rightarrow}
\newcommand{\MOLS}[2]{#1\operatorname{MOLS}(#2)}
\newcommand{\0}{\phantom{0}}

\newcommand{\sq}[1]{\parbox[c][1.4em][c]{1.4em}{\centering #1}}
\tikzset{element/.style = {minimum width=1.4em,minimum height=1.4em}}

\renewcommand{\orcidID}[1]{\unskip$^{[\text{#1}]}$}

\begin{document}

\title{Myrvold's Results on Orthogonal Triples of $10\times10$ Latin Squares: A SAT Investigation}

\titlerunning{Myrvold's Results on Orthogonal Triples of $10\times10$ Latin Squares}
\author{Curtis Bright\inst{1,2,3}\orcidID{0000-0002-0462-625X} \and
Amadou Keita\inst{2}\orcidID{0009-0001-5861-4617} \and
Brett Stevens\inst{3}\orcidID{0000-0003-4336-1773}}
\authorrunning{C. Bright et al.}
\institute{School of Computer Science, University of Waterloo,
Canada \\
 \and
Department of Mathematics and Statistics, University of Windsor,
Canada \\
\and
School of Mathematics and Statistics, Carleton University,
Canada \\
\email{cbright@uwaterloo.ca}, \email{keitaa@uwindsor.ca}, \email{brett@math.carleton.ca}
}
\maketitle
\begin{abstract}
Ever since E.~T.~Parker constructed an orthogonal pair
of $10\times10$ Latin squares in 1959,
an orthogonal triple of $10\times10$ Latin squares
has been one of the most sought-after combinatorial designs.
Despite extensive work, the existence of
such an orthogonal triple remains an open
problem, though some negative results are known.
In 1999, W.~Myrvold derived some highly restrictive constraints
in the special case in which
one of the Latin squares in the triple contains a $4\times4$ Latin subsquare.
In particular, Myrvold showed there were twenty-eight possible cases for an
orthogonal pair in such a triple,
twenty of which were removed from consideration.
We implement a computational approach that quickly verifies all of Myrvold's nonexistence results
and in the remaining eight cases finds explicit examples of orthogonal pairs---%
thus explaining for the first time why Myrvold's approach left eight cases unsolved.
As a consequence, the eight remaining cases cannot be removed by a strategy of
focusing on the existence of an orthogonal pair; the third square in the triple must necessarily
be considered as well.

Our approach uses a Boolean satisfiability (SAT) solver to
derive the nonexistence of twenty of the orthogonal pair types
and find explicit examples of orthogonal pairs in the eight remaining cases.
To reduce the existence problem into Boolean logic we
use a duality between the concepts
of \emph{transversal representation} and \emph{orthogonal pair}
and we provide a formulation of this duality in terms of a composition operation on Latin squares.
Using our SAT encoding,
we find transversal representations (and equivalently orthogonal pairs)
in the remaining eight cases in under two hours of computing on a
large computing cluster.
\end{abstract}

\section{Introduction}

A \emph{Latin square} of order~$n$ is an $n\times n$ array filled with $n$ distinct symbols,
usually taken to be $\{0,1,\dotsc,n-1\}$,
such that each symbol appears exactly once in each row and exactly once in each column. 
A \emph{transversal} of a Latin square of order~$n$ consists of~$n$ cells of the square 
chosen so that there is exactly one cell from each row, exactly one cell from each column, and 
exactly~$n$ distinct symbols all together.
There are many ways of representing a transversal, but we follow Myrvold~\cite{myrvold1999negative}
and represent a transversal by listing the symbols in the transversal in each column from left to right.
For example, the highlighted transversal in $\begin{bsmallmatrix}0&\textcolor{red}{\bf1}&2\\\textcolor{red}{\bf2}&0&1\\1&2&\textcolor{red}{\bf0}\end{bsmallmatrix}$
is represented by the row vector $[2,1,0]$.  We call this row vector the transversal's \emph{row representation}.

Two Latin squares $A$ and~$B$ of order~$n$ are said to be \emph{orthogonal} when all~$n^2$ possible symbol pairs
occur when the two squares are superimposed over each other.
This happens exactly when the~$n$ cell positions of the same symbol in $A$
form a transversal in~$B$ (regardless of the symbol chosen), thereby
decomposing~$B$ into~$n$ non-overlapping transversals.
A set of Latin squares that
are pairwise orthogonal to each other are known as \emph{mutually orthogonal Latin squares} (MOLS) and
a set of $k$ MOLS of order~$n$ are known as a $\MOLS{k}{n}$.
For each order $n$, let $N(n)$ denote the largest value of $k$
for which a $\MOLS{k}{n}$ exists.  Determining values of $N(n)$ has a long
history~\cite[Ch.~III]{abel2006mutually} and has been
of intense interest to mathematicians ever since Euler conjectured in 1782
that $N(n)=1$ for $n\equiv2\pmod4$.  It is easily seen that $N(2)=1$,
and Tarry showed in 1900 that $N(6)=1$~\cite{zbMATH02662385}.
However, in 1959, Euler's conjecture
was shown to be false by the discovery of a $\MOLS{2}{22}$~\cite{bose1959falsity}
and a $\MOLS{2}{10}$~\cite{parker1959orthogonal}.  In fact, in 1960
it was shown that $N(n)\geq2$
for all $n>6$~\cite{bose1960further}.
It is also known that $N(n)=n-1$ if and only if a projective plane
of order~$n$ exists.  Projective planes exist for all prime
powers, so the first order for which the value of $N(n)$ is uncertain is
$n=10$.  It is unknown if $N(10)\geq3$, and
determining the value of $N(10)$ is one of the
most prominent unsolved problems concerning MOLS\@.
In particular, finding a $\MOLS{3}{10}$ or proving its
nonexistence is a longstanding open problem in combinatorial design theory.

Although it is not known if a $\MOLS{3}{10}$ exists or not,
there are several special results known about this case.
Mann~\cite{mann1944orthogonal} proved that a $10\times10$ Latin square
with a $5\times5$ Latin subsquare cannot belong to an orthogonal pair, let alone an orthogonal triple.
Parker~\cite{parker1962orthogonal} proved that two orthogonal $10\times10$ Latin squares 
with orthogonal $3\times3$ Latin subsquares cannot be part of an orthogonal triple.
Myrvold~\cite{myrvold1999negative} considered a $10\times10$ Latin square $L$ with a $4\times4$ Latin subsquare.
She showed that it \emph{is} possible for $L$ to be part of an orthogonal pair, and further considered
if $L$ can be part of an orthogonal triple.
Myrvold showed
that orthogonal mates of $L$ can be classified into seven possible mate pattern types.
Furthermore, if $L$ is in an orthogonal triple the other two squares in the triple
can be classified into twenty-eight mate pattern type pairs.
Myrvold ruled out the existence of twenty of the twenty-eight mate pattern type pairs,
and this required only the consideration of constraints arising from two of the three putative squares.
Her work left open the remaining eight cases:
\begin{quote}
\emph{The most obvious next step in extending the current work is to eliminate the remaining eight cases from consideration}.~\cite{myrvold1999negative}
\end{quote}
We provide a reason why Myrvold's method was unable to rule out these eight cases,
and show any argument ruling out these cases must necessarily be more involved---because
orthogonal pairs in the remaining eight cases exist (though
it is unclear if orthogonal \emph{triples} in the remaining eight cases exist).
Thus, any argument ruling out the remaining eight cases must necessarily
involve the triple as a whole, not only two of the three squares.
We give more background on Latin squares
and the formulation of Myrvold's twenty-eight cases in Section~\ref{sec:background}.

Our approach uses a satisfiability (SAT) solver to explicitly construct a $\MOLS{2}{10}$
in each of the eight cases that Myrvold left open.  Additionally, in under a second of compute time
the SAT solver shows the nonexistence of a $\MOLS{2}{10}$ in the twenty cases solved by Myrvold.
To use a SAT solver, it is necessary to reduce the problem of searching
for the object in question to the problem of searching for a satisfying assignment
to a formula in Boolean logic representing Myrvold's framework and cases.

We reduce the problem of finding a $\MOLS{2}{10}$
in each of Myrvold's twenty-eight cases to SAT---see Section~\ref{sec:encoding}
for a description of our encoding.
We develop a SAT encoding of orthogonality that relies
on an equivalence between the orthogonality of Latin squares and what Myrvold calls a
``transversal representation'' Latin square~\cite{myrvold1999negative}.
Myrvold uses this equivalence for \emph{``designing computer programs for exploring squares
and their mates''}.  We provide a precise duality relating these two
concepts via a composition operation on Latin squares and
a generalization of Latin squares where only the columns (and not necessarily the rows)
contain all~$n$ symbols (see Section~\ref{sec:duality}).
This transversal representation encoding allowed
finding a $\MOLS{2}{10}$ for all of Myrvold's previously unsolved cases
in a reasonable amount of computation, even for a single desktop computer.
By exploiting the parallelization ability of a large computing cluster, we were able to solve
the hardest of the eight cases in less than
two hours of real time---see Section~\ref{sec:results} for more details.

\section{Background}\label{sec:background}

We define the notion of transversal representation and relate it
to the orthogonality of Latin squares in Section~\ref{sec:trans_ortho}.
Next, we explain the transversal representation types classified by Myrvold~\cite{myrvold1999negative} in Section~\ref{sec:trans_rep_types}, and
give a brief description of satisfiability solving in Section~\ref{sec:sat}.
Lastly, we give a summary of related work in Section~\ref{sec:related_work},
with a focus on work applying automated reasoning tools to solve problems
related to Latin squares.

\subsection{Transversals and Orthogonality}\label{sec:trans_ortho}

It is well-known that a Latin square of order~$n$ has an orthogonal mate
if and only if it can be decomposed into~$n$ disjoint transversals~\cite{Wanless2011}.
From the~$n$ disjoint transversals, a new Latin square can be formed by writing each transversal in its row representation
and stacking the rows together.
We call such a square a \emph{transversal representation} of the original square.  An
example of a $4\times4$ Latin square~$D$ with four
disjoint transversals and the associated transversal representation $D'$
is provided in Figure~\ref{fig:pairMN}.
The pair $(D,D')$ is known as a \emph{transversal representation pair} or \emph{TRP}.
\begin{figure}
\centering
\begin{tikzpicture}
\matrix (m) [matrix of nodes,nodes={element},column sep=-\pgflinewidth, row sep=-\pgflinewidth, label=left: \text{$D=$} ]{ 
|[draw,fill=black!70!green]| \textcolor{white}{1} &
|[draw,fill=yellow]| \textcolor{black}{2} &
|[draw,fill=white]| \textcolor{black}{0} &
|[draw,fill=light]| \textcolor{black}{3} \\
|[draw,fill=light]| \textcolor{black}{0} &
|[draw,fill=white]| \textcolor{black}{3} &
|[draw,fill=yellow]| \textcolor{black}{1} &
|[draw,fill=black!70!green]| \textcolor{white}{2} \\
|[draw,fill=white]| \textcolor{black}{2} &
|[draw,fill=light]| \textcolor{black}{1} &
|[draw,fill=black!70!green]| \textcolor{white}{3} &
|[draw,fill=yellow]| \textcolor{black}{0} \\
|[draw,fill=yellow]| \textcolor{black}{3} &
|[draw,fill=black!70!green]| \textcolor{white}{0} &
|[draw,fill=light]| \textcolor{black}{2} &
|[draw,fill=white]| \textcolor{black}{1} \\
 };
\end{tikzpicture} 
\begin{tikzpicture}
\matrix (m) [matrix of nodes,nodes={element},column sep=-\pgflinewidth, row sep=-\pgflinewidth, label=left: \text{$D'=$} ]{ 
|[draw,fill=light]| \textcolor{black}{0} &
|[draw,fill=light]| \textcolor{black}{1} &
|[draw,fill=light]| \textcolor{black}{2} &
|[draw,fill=light]| \textcolor{black}{3} \\
|[draw,fill=black!70!green]| \textcolor{white}{1} &
|[draw,fill=black!70!green]| \textcolor{white}{0} &
|[draw,fill=black!70!green]| \textcolor{white}{3} &
|[draw,fill=black!70!green]| \textcolor{white}{2} \\
|[draw,fill=white]| 2 &
|[draw,fill=white]| 3 &
|[draw,fill=white]| 0 &
|[draw,fill=white]| 1 \\
|[draw,fill=yellow]| \textcolor{black}{3} &
|[draw,fill=yellow]| \textcolor{black}{2} &
|[draw,fill=yellow]| \textcolor{black}{1} &
|[draw,fill=yellow]| \textcolor{black}{0} \\
 };
\end{tikzpicture} 
\caption{A transversal representation pair of Latin squares of order four.
Each transversal of $D$ is highlighted in a different colour, and the
row representations of the transversals are given in $D'$.}
\label{fig:pairMN}
\end{figure}

Although we are primarily interested in Latin squares, in the course
of our investigations, we found that it was helpful to consider the
more general case of column-Latin squares.
A \emph{column-Latin} square of order~$n$ is an $n\times n$ array filled with~$n$ distinct symbols
and in which each column contains distinct symbols (and is thus a permutation), but the rows
are not required to contain distinct symbols.
\emph{Row-Latin} squares are defined similarly: the rows of the square must contain distinct
entries, but the columns might not~\cite{laywine1998discrete}.
It follows immediately that an $n\times n$ array filled with~$n$ distinct symbols is a Latin square if and only if it is both
row-Latin and column-Latin.
For our purposes, the usefulness of column-Latin squares stems from the fact that
two column-Latin squares can be composed in a sensible way to form a third column-Latin square which preserves structure related to orthogonality (see Section~\ref{sec:duality}).
Thus, we state most of our results in terms of column-Latin squares.

The concept of orthogonality of Latin squares translates directly to column-Latin squares.
However, the concept of transversal needs some modification.  A generalized transversal
of a column-Latin square of order~$n$ must still be a selection of~$n$ entries
from each row and column, but the entries may not all be distinct.
Figure~\ref{fig:pairMN2} shows an example of this generalization;
note the generalized transversals highlighted in~$D_1$ contain duplicate entries
and therefore are not traditional transversals.  However, the row representation
construction
can still be used to construct the column-Latin square $D'_1$
and we refer to the pair $(D_1,D'_1)$ as a transversal representation pair
of column-Latin squares.
\begin{figure}
\centering
\begin{tikzpicture}
\matrix (m) [matrix of nodes,nodes={element},column sep=-\pgflinewidth, row sep=-\pgflinewidth, label=left: \text{$D_1=$} ]{ 
|[draw,fill=light]| \textcolor{black}{0} &
|[draw,fill=black!70!green]| \textcolor{white}{1} &
|[draw,fill=white]| \textcolor{black}{3} &
|[draw,fill=yellow]| \textcolor{black}{2} \\
|[draw,fill=black!70!green]| \textcolor{white}{1} &
|[draw,fill=yellow]| \textcolor{black}{3} &
|[draw,fill=light]| \textcolor{black}{2} &
|[draw,fill=white]| \textcolor{black}{0} \\
|[draw,fill=yellow]| \textcolor{black}{3} &
|[draw,fill=white]| \textcolor{black}{2} &
|[draw,fill=black!70!green]| \textcolor{white}{1} &
|[draw,fill=light]| \textcolor{black}{1} \\
|[draw,fill=white]| \textcolor{black}{2} &
|[draw,fill=light]| \textcolor{black}{0} &
|[draw,fill=yellow]| \textcolor{black}{0} &
|[draw,fill=black!70!green]| \textcolor{white}{3} \\
};
\end{tikzpicture} 
\begin{tikzpicture}
\matrix (m) [matrix of nodes,nodes={element},column sep=-\pgflinewidth, row sep=-\pgflinewidth, label=left: \text{$D'_1=$} ]{ 
|[draw,fill=light]| \textcolor{black}{0} &
|[draw,fill=light]| \textcolor{black}{0} &
|[draw,fill=light]| \textcolor{black}{2} &
|[draw,fill=light]| \textcolor{black}{1} \\
|[draw,fill=black!70!green]| \textcolor{white}{1} &
|[draw,fill=black!70!green]| \textcolor{white}{1} &
|[draw,fill=black!70!green]| \textcolor{white}{1} &
|[draw,fill=black!70!green]| \textcolor{white}{3} \\
|[draw,fill=white]| 2 &
|[draw,fill=white]| 2 &
|[draw,fill=white]| 3 &
|[draw,fill=white]| 0 \\
|[draw,fill=yellow]| \textcolor{black}{3} &
|[draw,fill=yellow]| \textcolor{black}{3} &
|[draw,fill=yellow]| \textcolor{black}{0} &
|[draw,fill=yellow]| \textcolor{black}{2} \\
 };
\end{tikzpicture} 
\caption{A transversal representation pair of $4\times4$ column-Latin squares.
Note that the highlighted entries of $D_1$ are \emph{not} transversals, but
their row representations when placed in a $4\times4$ array do form a
column-Latin square.}
\label{fig:pairMN2}
\end{figure}

We now give purely logical definitions of orthogonal pair
and transversal representation and state the definitions in a way
that highlights the similarity between the concepts.
Suppose $[a_0,\dotsc,a_{n-1}]$ is a row representing a generalized transversal of a column-Latin square~$B$.
This means if $i$ is a row index, $j$ and~$j'$ are two distinct column indices, and
$B[i,j]=a_j$, then $B[i,j']\neq a_{j'}$ (otherwise, both the $j$th and $j'$th
entries of the generalized transversal are in row~$i$, which is not allowed in any transversal, generalized or not).
Equivalently, if both $B[i,j]=a_j$ and $B[i,j']=a_{j'}$, then the only possibility is that $j=j'$.
This motivates the following definition.

\begin{definition}\label{def:transrep}
Let $A$ and $B$ be order~$n$ column-Latin squares.
Row\/ $i$ of $A$ \emph{represents a transversal} of~$B$ when
$A[i, j] = B[i', j]$ and $A[i, j'] = B[i', j']$ imply $j = j'$.
The square $A$ is said to be a \emph{transversal representation} of\/~$B$ when
each row of $A$ represents a transversal of $B$, i.e.,
for all\/ $0\leq i,i',j,j'<n$,
\[ A[i, j] = B[i', j] \text{ and\/ } A[i, j'] = B[i', j'] \text{ imply } j = j' . \]
\end{definition}

Because Definition~\ref{def:transrep} is symmetric in $A$ and~$B$, $A$ is a transversal representation of $B$
if and only if $B$ is a transversal representation of~$A$.
As before, we say $(A,B)$ is a \emph{transversal representation pair} or \emph{TRP}.

On the other hand, if two column-Latin squares $A$ and $B$ are orthogonal this means that
if $(i,j)$ and $(i',j')$ are two distinct (row, column) pairs then
$(A[i,j],B[i,j])\neq(A[i',j'],B[i',j'])$.  Equivalently, it means that if both
$A[i,j]=A[i',j']$ and $B[i,j]=B[i',j']$, the only possibility is that $(i,j)=(i',j')$.
This motivates the following definition.

\begin{definition}\label{def:ortho}
Let $A$ and $B$ be order~$n$ column-Latin squares.
$A$ is said to be \emph{orthogonal} to $B$ if
for all\/ $0\leq i,i',j,j'<n$,
\[ A[i,j] = A[i',j'] \text{ and\/ } B[i,j] = B[i',j'] \text{ imply } j = j' . \]
\end{definition}

Note that the equality of $j$ and~$j'$ in Definition~\ref{def:ortho} also implies the equality of $i$ and~$i'$
because $A$ and~$B$ are column-Latin squares.
The consequent in Definition~\ref{def:ortho} thus could equivalently have
been written as the more typical
$(i,j)=(i',j')$, but we use the simpler $j=j'$
in order to highlight the striking similarity between Definitions~\ref{def:transrep} and~\ref{def:ortho}.

\subsection{Transversal Representation Types}\label{sec:trans_rep_types}

We now review Myrvold's results~\cite{myrvold1999negative}
on the possible transversal representation types of a $10\times10$ Latin square~$L$
containing a $4\times4$ Latin subsquare~$\Omega$.
Without loss of generality, we assume the subsquare appears in the bottom-right of~$L$,
i.e., in the rows and columns labeled 6 to 9.  We also assume
$L$ consists of the symbols from the set $\{0,1,2,\dotsc,9\}$
and $\Omega$ consists of symbols from the set $\{0,1,2,3\}$.
We partition the other regions of~$L$ into~$\Delta$ (lower-left), $\Gamma$ (upper-right), and $\Sigma$ (upper-left) as shown in Figure~\ref{fig:decomp}.
Since the subsquare $\Omega$ is a Latin square containing symbols from the set $\{0,1,2,3\}$, the rectangles $\Delta$ and~$\Gamma$
must take symbols only from the set $\{4,5,6,\dotsc,9\}$ and each row and column of $\Sigma$ must contain exactly $6-4=2$ symbols
from the set $\{4,5,6,\dotsc,9\}$.

\begin{figure}
  \subfloat{
	\begin{minipage}[c][1\width]{
	   0.48\textwidth}
	   \centering
	  \begin{tikzpicture}[element/.style={minimum width=0.5cm,minimum height=0.5cm}]
\matrix (m) [matrix of nodes,nodes={element},column sep=-\pgflinewidth, row sep=-\pgflinewidth, ampersand replacement=\&]{ 
|[draw,fill=white ]| \& |[draw,fill=white ]| \& |[draw,fill=white ]| \& |[draw,fill=white ]| \& |[draw,fill=black!70!green]| \& |[draw,fill=black!70!green]| \& |[draw,fill=light ]| \& |[draw,fill=light ]| \& |[draw,fill=light ]| \& |[draw,fill=light ]|\\
|[draw,fill=white ]| \& |[draw,fill=white ]| \& |[draw,fill=white ]| \& |[draw,fill=white ]| \& |[draw,fill=black!70!green]| \& |[draw,fill=black!70!green]| \& |[draw,fill=light ]| \& |[draw,fill=light ]| \& |[draw,fill=light ]| \& |[draw,fill=light ]| \\
|[draw,fill=white ]| \& |[draw,fill=white ]| \& |[draw,fill=white ]| \& |[draw,fill=white ]| \& |[draw,fill=black!70!green]| \& |[draw,fill=black!70!green]| \& |[draw,fill=light ]| \& |[draw,fill=light ]| \& |[draw,fill=light ]| \& |[draw,fill=light ]| \\
|[draw,fill=white ]| \& |[draw,fill=white ]| \& |[draw,fill=white ]| \& |[draw,fill=white ]| \& |[draw,fill=black!70!green]| \& |[draw,fill=black!70!green]| \& |[draw,fill=light ]| \& |[draw,fill=light ]| \& |[draw,fill=light ]| \& |[draw,fill=light ]| \\
|[draw,fill=white ]| \& |[draw,fill=white ]| \& |[draw,fill=white ]| \& |[draw,fill=white ]| \& |[draw,fill=black!70!green]| \& |[draw,fill=black!70!green]| \& |[draw,fill=light ]| \& |[draw,fill=light ]| \& |[draw,fill=light ]| \& |[draw,fill=light ]| \\
|[draw,fill=white ]| \& |[draw,fill=white ]| \& |[draw,fill=white ]| \& |[draw,fill=white ]| \& |[draw,fill=black!70!green]| \& |[draw,fill=black!70!green]| \& |[draw,fill=light ]| \& |[draw,fill=light ]| \& |[draw,fill=light ]| \& |[draw,fill=light ]|\\
|[draw,fill=light ]| \& |[draw,fill=light ]| \& |[draw,fill=light ]| \& |[draw,fill=light ]| \& |[draw,fill=light ]| \& |[draw,fill=light ]| \& |[draw,fill=white ]| \& |[draw,fill=white ]| \& |[draw,fill=white ]| \& |[draw,fill=white ]|\\
|[draw,fill=light ]| \& |[draw,fill=light ]| \& |[draw,fill=light ]| \& |[draw,fill=light ]| \& |[draw,fill=light ]| \& |[draw,fill=light ]| \& |[draw,fill=white ]| \& |[draw,fill=white ]| \& |[draw,fill=white ]| \& |[draw,fill=white ]|\\
|[draw,fill=light ]| \& |[draw,fill=light ]| \& |[draw,fill=light ]| \& |[draw,fill=light ]| \& |[draw,fill=light ]| \& |[draw,fill=light ]| \& |[draw,fill=white ]| \& |[draw,fill=white ]| \& |[draw,fill=white ]| \& |[draw,fill=white ]| \\
|[draw,fill=light ]| \& |[draw,fill=light ]| \& |[draw,fill=light ]| \& |[draw,fill=light ]| \& |[draw,fill=light ]| \& |[draw,fill=light ]| \& |[draw,fill=white ]| \& |[draw,fill=white ]| \& |[draw,fill=white ]| \& |[draw,fill=white ]| \\
 };
\draw[line width=4pt, black, -] (m-1-7.north west) -- (m-10-7.south west);
\draw[thick,black] (1.5,-1.5) node[transform shape, scale=3] {$\Omega$};
\draw[thick,black] (-1.0,-1.5) node[transform shape, scale=3] {$\Delta$};
\draw[thick,black] (1.5,1.0) node[transform shape, scale=3] {$\Gamma$};
\draw[thick,black] (-1.0,1.0) node[transform shape, scale=3] {$\Sigma$};
\end{tikzpicture}
	\end{minipage}}
 \hfill 	
  \subfloat{
	\begin{minipage}[c][1\width]{
	   0.48\textwidth}
	   \centering
	   \begin{tikzpicture}[element/.style={minimum width=0.5cm,minimum height=0.5cm}]
\matrix (m) [matrix of nodes,nodes={element},column sep=-\pgflinewidth, row sep=-\pgflinewidth, ampersand replacement=\&]{ 
$p_0$: \& |[draw,fill=white ]| \& |[draw,fill=white ]| \& |[draw,fill=white ]| \& |[draw,fill=white ]| \& |[draw,fill=light ]| \& |[draw,fill=light ]| \& |[draw,fill=light ]| \& |[draw,fill=light ]| \& |[draw,fill=light ]| \& |[draw,fill=light ]|\\
 };
\draw[line width=4pt, black, -] (m-1-7.north east) -- (m-1-7.south east);
\end{tikzpicture} 
	   \begin{tikzpicture}[element/.style={minimum width=0.5cm,minimum height=0.5cm}]
\matrix (m) [matrix of nodes,nodes={element},column sep=-\pgflinewidth, row sep=-\pgflinewidth, ampersand replacement=\&]{ 
$p_1$: \& |[draw,fill=white ]| \& |[draw,fill=white ]| \& |[draw,fill=white ]| \& |[draw,fill=light ]| \& |[draw,fill=light ]| \& |[draw,fill=light ]| \& |[draw,fill=white ]| \& |[draw,fill=light ]| \& |[draw,fill=light ]| \& |[draw,fill=light ]| \\
 };
\draw[line width=4pt, black, -] (m-1-7.north east) -- (m-1-7.south east);
\end{tikzpicture} 
	   \begin{tikzpicture}[element/.style={minimum width=0.5cm,minimum height=0.5cm}]
\matrix (m) [matrix of nodes,nodes={element},column sep=-\pgflinewidth, row sep=-\pgflinewidth, ampersand replacement=\&]{ 
$p_2$: \& |[draw,fill=white ]| \& |[draw,fill=white ]| \& |[draw,fill=light ]| \& |[draw,fill=light ]| \& |[draw,fill=black!70!green]| \& |[draw,fill=black!70!green]| \& |[draw,fill=white ]| \& |[draw,fill=white ]| \& |[draw,fill=light ]| \& |[draw,fill=light ]| \\
 };
\draw[line width=4pt, black, -] (m-1-7.north east) -- (m-1-7.south east);
\end{tikzpicture} 
	   \begin{tikzpicture}[element/.style={minimum width=0.5cm,minimum height=0.5cm}]
\matrix (m) [matrix of nodes,nodes={element},column sep=-\pgflinewidth, row sep=-\pgflinewidth, ampersand replacement=\&]{ 
$p_3$: \& |[draw,fill=white ]| \& |[draw,fill=light ]| \& |[draw,fill=black!70!green]| \& |[draw,fill=black!70!green]| \& |[draw,fill=black!70!green]| \& |[draw,fill=black!70!green]| \& |[draw,fill=white ]| \& |[draw,fill=white ]| \& |[draw,fill=white ]| \& |[draw,fill=light ]| \\
 };
\draw[line width=4pt, black, -] (m-1-7.north east) -- (m-1-7.south east);
\end{tikzpicture} 
	   \begin{tikzpicture}[element/.style={minimum width=0.5cm,minimum height=0.5cm}]
\matrix (m) [matrix of nodes,nodes={element},column sep=-\pgflinewidth, row sep=-\pgflinewidth, ampersand replacement=\&]{ 
$p_4$: \& |[draw,fill=black!70!green]| \& |[draw,fill=black!70!green]| \& |[draw,fill=black!70!green]| \& |[draw,fill=black!70!green]| \& |[draw,fill=black!70!green]| \& |[draw,fill=black!70!green]| \& |[draw,fill=white ]| \& |[draw,fill=white ]| \& |[draw,fill=white ]| \& |[draw,fill=white ]| \\
 };
\draw[line width=4pt, black, -] (m-1-7.north east) -- (m-1-7.south east);
\end{tikzpicture}  \\
	\end{minipage}}
	\caption{The Latin square $L$ (left) and its possible transversal types (right).
	White cells represent symbols in $\{0,1,2,3\}$, light cells represent
	symbols in the rectangles $\Delta$ and $\Gamma$, 
	and dark cells represent the symbols $\{4,5,\dotsc,9\}$ in $\Sigma$.
	The cells of $\Sigma$ are not shown in
	absolute positions; in actuality, each row and column of $\Sigma$ has exactly two
	dark cells.  Similarly, the transversal types are shown up to a permutation
	of the first six entries and the last four entries.}
    \label{fig:decomp}
\end{figure}

Suppose the cells with symbols in $\{0,1,2,3\}$ are coloured white.
A transversal of $L$ can be of five possible forms depending on
how many white cells it takes from the Latin subsquare~$\Omega$.
A transversal containing $i$ white cells from $\Omega$ (i.e., in its last four columns)
is said to be of form $p_i$ (see Figure~\ref{fig:decomp}).  
Since any transversal will contain
exactly four white cells in total, it must contain $4-i$ white cells
in its first six columns. Consider the entries of $p_i$ that were chosen from the first six rows of~$L$ (i.e., $\Sigma$ or $\Gamma$).
We have $4-i$ white entries (all from $\Sigma$) and $4-i$
entries from the last four columns of $L$ (i.e., from $\Gamma$), so there are $6-2(4-i)=2i-2$ remaining entries.
The only possibilities for these are the nonwhite entries of $\Sigma$, and we colour these entries dark.
This results in the following lemma.
\begin{lemma}[{\cite[Lemma~3.1]{myrvold1999negative}}]\label{lem:darkcount}
A transversal of type $p_i$ contains exactly $2i-2$ dark entries.
\end{lemma}
A simple corollary of Lemma~\ref{lem:darkcount} is that $p_0$ is not a possible type,
as it would have to contain $-2$ dark entries.

Let $n_i$ be the number of transversals of type $p_i$ in a transversal representation of $L$. Simple counting arguments give that the values $\{n_1,n_2,n_3,n_4\}$ satisfy the following Diophantine linear system.
\begin{align*}
    n_i&\geq 0 && \text{ nonnegativity of the counts,} \\
    n_1 + n_2 + n_3 + n_4 &= 10 && \text{ ten total transversals,} \\
    n_1 + 2n_2 + 3n_3 + 4n_4 &= 16 && \text{ sixteen total symbols in $\Omega$.}
\end{align*}
There are seven possible solutions to this linear system and correspondingly seven transversal representation types of~$L$.
These types are denoted R, S, T, U, V, W, and X by Myrvold.
Table~\ref{tbl:transtypes} gives the transversal type counts of each case.

\begin{table}
\caption{A summary of Myrvold's seven possible transversal types of $L$.}\label{tbl:transtypes}%
\begin{center}
\begin{tabular}{ccccc}
Type & $n_1$ & $n_2$ & $n_3$ & $n_4$ \\ \hline
R    &8 &0 &0 &2 \\ 
S    &7 &0 &3 &0 \\ 
T    &7 &1 &1 &1 \\ 
U    &6 &2 &2 &0 \\ 
V    &6 &3 &0 &1 \\ 
W    &5 &4 &1 &0 \\ 
X    &4 &6 &0 &0 
\end{tabular}
\end{center}\vspace{-1em}
\end{table}

Up to ordering, there are $\binom{7}{2}=21$ ways of choosing a pair with two different types,
and $7$ ways of choosing a pair with matching types,
for a total of $28$ possible transversal representation
pair combinations.  Under the assumption that~$L$ is part of an orthogonal triple,
Myrvold~\cite[Thm 4.4]{myrvold1999negative} showed that the only possible pair types that could potentially be transversal representations
of~$L$ simultaneously are
$(\S,\X)$, $(\U,\U)$, $(\U,\W)$, $(\U,\X)$, $(\V,\X)$, $(\W,\W)$, $(\W,\X)$, and $(\X,\X)$.

\subsection{Satisfiability Solving}\label{sec:sat}

In this section, we provide some basic preliminaries on Boolean logic
and satisfiability (SAT) solving.  A \emph{SAT solver} is a program that can determine
if a Boolean logic formula can be \emph{satisfied}---that is, if there is a truth
assignment under which the formula becomes true.
In practice, the formulas provided to SAT solvers must be written in conjunctive normal form (CNF)\@.
Formulas in CNF only contain the Boolean connective operators $\land$ (and), $\lor$ (or), and $\lnot$ (not).
These operators have meanings similar to those in everyday English:
the formula $x\land y$ is true if and only if both $x$ and $y$ are true;
the formula $x\lor y$ is true if and only if $x$ or $y$ (or both) are true;
and the formula $\lnot x$ is true if and only if $x$ is false.

A \emph{literal} is a Boolean variable or its negation, i.e., a formula of the form $x$ or $\lnot x$ where $x$ is a Boolean variable.
A \emph{clause} is a disjunction of literals, i.e., a formula of the form $l_1\lor\dotsb\lor l_k$ where $l_1$, $\dotsc$, $l_k$ are literals.
Finally, a formula is in \emph{conjunctive normal form} when it is a conjunction of clauses, i.e., a formula of the form $c_1\land\dotsb\land c_k$ where $c_1$, $\dotsc$, $c_k$ are clauses.

When $A$ is a conjunction of literals and $B$ is a disjunction of literals, we use the notation $A\limp B$
as shorthand for $\lnot A\lor B$.  By basic logic equivalences, the formula $(\lnot\bigwedge_i a_i)\lor\bigvee_i b_i$
is equivalent to $\bigvee_i\lnot a_i\lor\bigvee_i b_i$, which (after applying the simplification $\lnot\lnot x\equiv x$
to any doubly negated literal) is a clause.  Thus, we consider the notation $A\limp B$ to be shorthand for a clause
when $B$ is a clause and $A$ is a conjunction of literals.

Although there is no guarantee that SAT solvers can solve the SAT problem in
a feasible amount of time, modern SAT solvers are highly effective at solving
many kinds of problems arising in practice~\cite{Vardi2014},
including mathematical problems such as the Boolean Pythagorean triples
problem~\cite{Heule2016} and Lam's problem
of proving the nonexistence of a projective plane of order ten~\cite{bright2021sat}.
Although these problems at first seem unconnected to logic, they can
be reduced to SAT due to the versatility of Boolean logic~\cite{bright2020effective}.
Another advantage of using a SAT solver is that they offer
a higher amount of confidence in a computational search.  It is typically less error-prone
to write a SAT encoding than it is to write optimized search code, and moreover, the
SAT solver itself does not need to be trusted because it produces a proof certificate
which can be later checked by simpler and independently-written software.
This is particularly
relevant when purporting to demonstrate
the \emph{nonexistence} of a mathematical object, such as in
Lam's problem of
proving projective planes of order ten do not exist~\cite{Lam1991}.

Lam's problem was resolved in 1989 using a massive computer search
by Lam, Thiel, and Swiercz~\cite{lam1989non}.
In 2011, the search was independently performed by Roy~\cite{roy2011confirmation}.
Although these works are amazing achievements, they both crucially rely on highly optimized
computer code that is essentially impossible to verify for correctness, and the programmers
of the search code were upfront that the code may contain bugs.  Indeed,
discrepancies in the results of these searches were later found: a SAT-based search of Bright et al.~\cite{bright2021sat} found
inconsistencies in the intermediate counts provided by Lam et al., implying a small
number of missing subcases in the proof.
Also, the independent
confirmation of Roy~\cite{roy2011confirmation} was based in part on the nonexistence
of a partial projective plane later determined to actually exist~\cite{Bright2020b}.
There is no formal proof that Bright et al.'s SAT-based resolution
of Lam's problem is without error---because the SAT encoding itself is unverified---%
but it does have the advantage that no \emph{search code} has to be trusted.

\subsection{Related Work}\label{sec:related_work}

Extensive searches for a $\MOLS{3}{10}$
have been performed, and some important cases have been ruled out.  For example,
it is known that any such triple must only contain Latin squares with trivial symmetry groups~\cite{mckay2007small}.
Independent computer searches~\cite{bright2021sat,lam1989non,roy2011confirmation}
have revealed that there is no projective plane of order ten,
and because a projective plane of order $n$ is equivalent to a $\MOLS{(n-1)}{n}$~\cite{Bose1938,Moore1896},
these searches imply that no $\MOLS{9}{10}$s exist or equivalently that $N(10)<9$.
Together with a result of Bruck~\cite{bruck1963finite}, this implies
that $N(10)\leq 6$ which is currently the best upper bound known on $N(10)$.

Egan and Wanless~\cite{Egan2015} enumerate MOLS of small orders, providing counts of 
orthogonal mates and classifications up to various equivalence notions for orders $n \leq 9$. 
They also present a set of three Latin squares $L_1$, $L_2$, $L_3$ of order 10 that is the closest known to forming a 
complete set of MOLS: $L_1$ is orthogonal to both $L_2$ and $L_3$,
and $91$ out of the $100$ symbol pairs are different when $L_2$ and $L_3$ are superimposed.
They also showed that $L_2$ and $L_3$ have seven common disjoint transversals.

Numerous studies have leveraged SAT solving, integer programming, and constraint programming
in order to search for Latin squares of various forms.
Appa, Magos, and Mourtos~\cite{appa2006searching,appa2002integrating} integrated integer programming and 
constraint programming to tackle the problem of searching for mutually orthogonal Latin squares.
Their comparative study against traditional constraint and integer programming algorithms revealed the 
effectiveness of combining integer and constraint programming in searching for $\MOLS{2}{n}$ for $n\leq12$ and 
$\MOLS{3}{n}$ for $n\leq9$.
Rubin et al.~\cite{rubin2021integer}
formulated a symmetry breaking method and also provided an alternative constraint programming
encoding based on a theorem of Mann~\cite{mann1942construction}
which performed much better in their search for pairs of orthogonal Latin squares.
The SAT encoding that we use in our work can be viewed as a
reformulation of their constraint programming encoding
into Boolean satisfiability.

Ma and Zhang~\cite{ma2013finding} use a general-purpose model searching program 
to find MOLS\@. 
They show a $\MOLS{k}{n}$ exists if and only if there exists a Latin square of order $n$ which has $k - 1$ transversal matrices $T_1$, $\dotsc$, $T_{k-1}$ with any two transversal matrices $T_i$ and $T_j$ ($i \neq j$)
being transversal matrices of each other~\cite[Prop~1]{ma2013finding}.
As a result, instead of searching for $\MOLS{k}{n}$, 
they searched for $k$ Latin squares $L$, $T_1$, $\dotsc$, $T_{k-1}$
that are mutual transversal matrices of each other.
The initial Latin square $L$ was defined as a function $f \colon \mathcal{R} \times \mathcal{C} \to \mathcal{D}$
on row indices~$\mathcal{R}$, column indices $\mathcal{C}$, and symbol set $\mathcal{D}$.
Similarly, the $i$th transversal matrix $T_i$ ($1 \leq i \leq k-1$) was defined as a function
$f_i \colon \mathcal{D}_i \times \mathcal{C} \to \mathcal{R}$,
where $\mathcal{D}_i$ is the symbol set of $L_i$, the Latin square represented by the transversal matrix $T_i$.
The formulae they used for encoding a $\MOLS{k}{n}$ then consist of three types:
\begin{enumerate}
\item Formulae to specify that $f$ and $f_i$ are Latin squares:
	\begin{align*}
	f(x_1, y) = f(x_2, y) &\limp x_1 = x_2, &
	f(x, y_1) = f(x, y_2) &\limp y_1 = y_2, \\
	f_i(t_1, y) = f_i(t_2, y) &\limp t_1 = t_2, &
	f_i(t, y_1) = f_i(t, y_2) &\limp y_1 = y_2.
	\end{align*}
\item Formulae to specify that $f_i$ is a transversal matrix of $f$:
 \[ f(f_i(t, y_1), y_1) = f(f_i(t, y_2), y_2) \limp y_1 = y_2.\]
\item Formulae to ensure that $L_i$ and $L_j$ are orthogonal
by stating that $T_i$ and $T_j$ are a transversal representation pair:
\[ \bigl(f_i(t_1, y_1) = f_j (t_2, y_1) \land f_i(t_1, y_2) = f_j(t_2, y_2)\bigr) \limp y_1 = y_2 . \]
\end{enumerate}
Our encoding of a transversal representation pair uses formulae that are similar to their first two types,
though our encoding is purely represented as a Boolean satisfiability
problem which does not natively support expressions like $f(f_i(t, y_1), y_1)$.
Constraints of type~3 could theoretically be replaced by constraints like
those of type~2 (e.g., $f_i(f_j(t, y_1), y_1) = f_i(f_j(t, y_2), y_2) \limp y_1 = y_2$),
though it is unclear if this encoding variant was tried by Ma and Zhang.
Our experience suggests that (at least for a SAT solver)
it is preferable to encode
a transversal representation pair using constraints of type~2
instead of constraints of type~3.

A Latin square that is orthogonal to its transpose is known as \emph{self-orthogonal}
and if it is additionally orthogonal to its anti-diagonal transpose it
is known as \emph{doubly self-orthogonal}.
For orders $n\equiv2\pmod4$, the existence of doubly self-orthogonal Latin squares is unknown for $n > 10$.
In 2011, Lu et~al.~\cite{lu2011searching} proved the nonexistence of a doubly self-orthogonal Latin square of order ten.
They encoded the existence of a doubly self-orthogonal Latin square of order ten as a SAT problem
and proved the nonexistence by showing the resulting SAT instance was unsatisfiable.
To describe their encoding, let $A$ be a self-orthogonal Latin square of order $n$, let $A^T$ denote the transpose of $A$,
and let $A^*$ denote the transpose across the anti-diagonal of $A$, i.e.,  
\( A^T[x, y] = A[y, x] \) and \( A^*[x, y] = A[n - 1 - y, n - 1 - x] \) where $0\leq x,y<n$.
In addition to the properties of a Latin square, they generated the constraints
\begin{align*}
&(A[x_1, y_1] = A[x_2, y_2] \land A[y_1, x_1] = A[y_2, x_2]) \\
&\quad\limp (x_1 = x_2 \land y_1 = y_2), \quad\text{i.e., orthogonality of $A$ and $A^T$, and } \\
&(A[x_1, y_1] = A[x_2, y_2] \land A[n - 1 - y_1, n - 1 - x_1] = A[n - 1 - y_2, n - 1 - x_2]) \\
&\quad\limp (x_1 = x_2 \land y_1 = y_2), \quad\text{i.e., orthogonality of $A$ and $A^*$}.
\end{align*}

A \emph{Costas array} of order $n$ is an $n\times n$ grid with $n$ dots and $n^2 - n$ empty cells,
with one dot in every row and column, and with
no two dots sharing the same relative horizontal, vertical, or diagonal displacement.
A \emph{Costas Latin square} is a Latin square in which the cells for each symbol form a Costas array;
see Figure~\ref{fig:costas} for an example.
\begin{figure}\centering
\begin{tikzpicture}
\matrix (m) [matrix of nodes,nodes={element,draw},column sep=-\pgflinewidth, row sep=-\pgflinewidth]{ 
0 & 1 & 2 & 3 \\
1 & 0 & 3 & 2 \\
3 & 2 & 1 & 0 \\
2 & 3 & 0 & 1 \\
};
\end{tikzpicture}
\caption{An example $4\times4$ Costas Latin square.}\label{fig:costas}
\end{figure}
Jin et al.~\cite{jin2021investigating} used SAT solvers to search for Costas Latin squares.
They established new existence and nonexistence results for various types of Costas Latin squares of even orders $n \leq 10$
including orthogonal pairs of Costas Latin squares.
In their encoding, they define from the square~$A$ a new square $\mathit{TA}$ by the rule $A[i,j]=k\rightarrow\mathit{TA}[k,j]=i$.
This makes $\mathit{TA}$
the \emph{$(3,2,1)$-parastrophe} of~$A$ (the Latin square obtained by swapping the meaning of rows and symbols),
though they refer to $\mathit{TA}$ as a transversal matrix.
To encode orthogonality of $(A,B)$, they impose the constraints
\[ x \neq y \limp (\mathit{TA}[u, x] \neq \mathit{TB}[v, x] \lor \mathit{TA}[u, y] \neq \mathit{TB}[v, y]) \quad \text{for $0 \leq x,y,u,v < n$}. \]
The $(3,2,1)$-parastrophe is also called the \emph{column inverse} since it can also be obtained
by treating each column as a permutation of $[0,\dotsc,n-1]$ and replacing each column with
its inverse~\cite{keedwell}.
In the rest of this paper, we will use the notation $A^{-1}$ for the column inverse of $A$ (see Section~\ref{sec:composition}).

A Latin square of order~$n$ is \emph{idempotent} when its diagonal consists of the
entries $0$, $1$, $\dotsc$, $n-1$ in order, and is \emph{symmetric} when it is equal to its
own transpose.
A \emph{golf design} of order $n$ is a collection of $n - 2$ idempotent symmetric Latin squares of order~$n$ that are mutually disjoint,
meaning that any two Latin squares in the collection share no common symbols in any cell (except for the cells along their diagonals).
Two golf designs are \emph{orthogonal} if every Latin square in one design has an orthogonal mate in the other design.

Huang et al.~\cite{Huang2019} investigated the existence of orthogonal golf designs via constraint programming
and satisfiability testing.
They reformulated the orthogonal mate finding problem as a transversal finding problem.
They constructed the transversal matrix $T$ of a Latin square $L$ with the constraints
\[ \left( y_1 = y_2 \lor L[T[x, y_1],y_1] \neq L[T[x, y_2],y_2]\right) \quad\text{for $0\leq x,y_1,y_2<n$} , \]
and additionally used constraints specifying that $T$ is a Latin square.

Latin squares are known as \emph{diagonal} if they feature distinct symbols along both the main and back diagonals. 
Zaikin and Kochemazov~\cite{oleg2015search} constructed SAT encodings to discover pairs of orthogonal 
diagonal Latin squares of order ten and pseudotriples of orthogonal diagonal Latin squares. A \emph{pseudotriple} 
refers to a set of three Latin squares that nearly form an orthogonal triple, but the orthogonality
condition is only required to hold on a subset of the cells of the Latin squares.  They discovered a
triple of diagonal Latin squares of order ten for which the orthogonality condition holds across 73 cells
(the same 73 cells in each Latin square in the triple).

An \emph{extended self-orthogonal diagonal Latin square} is a diagonal Latin square
that is orthogonal to a diagonal Latin square in its main class---the
\emph{main class} of a Latin square being the set of Latin squares produced by
application of row permutations, column permutations, symbol permutations,
or interchanging the roles of rows, columns, and symbols.  Extended self-orthogonal
diagonal Latin squares generalize the notion of self-orthogonal diagonal Latin squares,
since the transpose of a Latin square is always a member of its main class
(obtained by interchanging the roles of rows and columns).  Zaikin, Vatutin, and
Bright~\cite{Zaikin2025} use a SAT solver to enumerate all extended self-orthogonal diagonal Latin squares
up to order ten and show that in order ten no such squares are part of an orthogonal triple.
Their SAT encoding for orthogonality is based off of the one we present in this paper
relying on a consequence of Mann's theorem described in Section~\ref{sec:composition}.

In a separate recently published paper~\cite{Bright2025}, we use a SAT encoding for orthogonality
based on the one described in Section~\ref{sec:TRPC} in order to enumerate
$\MOLS{2}{10}$ whose incidence matrices have at least two nontrivial linear
dependencies.  This enumeration had been previously completed using custom-written search code
of Delisle~\cite{delisle2010search} and was motivated by work of Dukes and Howard~\cite{Dukes2012}
which classified the kinds of linear dependencies that could occur in the incidence
matrix of a hypothetical set of $\MOLS{4}{10}$.  Dukes and Howard also showed that the incidence matrix
of a $\MOLS{4}{10}$ must have at least two nontrivial linear dependencies.
Based on a later computational search of Gill and Wanless~\cite{Gill2023}, it is now known
that the incidence matrix of any pair of squares in a $\MOLS{3}{10}$ must only have trivial
linear dependencies.  Consequently, the rank of the linear code generated
by any pair of squares in a $\MOLS{3}{10}$ must be exactly 37.

\section{Composition and Duality}\label{sec:duality}

In this section, we describe a duality between the concepts of
orthogonality and transversal representation.  First, in Section~\ref{sec:composition}
we define a composition operation on column-Latin squares.
Then in Section~\ref{sec:op_tr} we use
the composition operation to concisely characterize the duality.

\subsection{Composition of Column-Latin Squares}\label{sec:composition}

A column-Latin square of order $n$ can be represented by $(c_0, c_1, \dotsc, c_{n-1})$ where~$c_j$ is the permutation of $[0,\dotsc,n-1]$ formed by the $j$th column.
For any two permutations $f$ and~$g$ on the same set, the \emph{composition} $fg$ is another permutation where $(fg)(i) = f(g(i))$, i.e., applying $g$ then $f$.
The composition of two column-Latin squares $F=(f_0, \dotsc, f_{n-1})$ and $G=(g_0, \dotsc, g_{n-1})$ is defined as 
\[ FG = (f_0g_0,\dotsc,f_{n-1}g_{n-1}). \]
The $(i,j)$th entry of $FG$ is then $f_jg_j(i)=F[G[i,j],j]$.
The \emph{column inverse} of a column-Latin square~$F$, denoted $F^{-1}$,
is the column-Latin square in which each column is the inverse permutation of the corresponding column of~$F$. 

Let $e$ denote the identity column permutation with $e(i)= i$ for $0\leq i<n$ and
$E = (e,\dotsc,e)$ the column-Latin square of order $n$ formed by $n$ copies of $e$.
The following two lemmas appear in Laywine and Mullen~\cite[pp.~98--99]{laywine1998discrete},
except stated in terms of row-Latin squares instead of column-Latin squares.

\begin{lemma}\label{A and E}
Let\/ $C$ be a column-Latin square. Then $(C,E)$ is an orthogonal pair if and only if\/ $C$ is a Latin square.
\end{lemma}

\begin{lemma}\label{orthogonal product}
If\/ $\{C_1, C_2, \dotsc, C_m\}$ is a set of mutually orthogonal column-Latin squares, then for any column-Latin square $G$, the set $\{C_1G, C_2G, \dotsc, C_mG\}$ comprises a set of mutually orthogonal column-Latin squares.
\end{lemma}

The next proposition provides criteria establishing a necessary and sufficient condition for
the orthogonality of two column-Latin squares.  In particular, the existence of a Latin square of a certain
form guarantees the orthogonality of the two column-Latin squares.
The biconditional statement in the proposition
was proven by Mann~\cite{mann1942construction} and also appears in Norton~\cite[Thm.~2]{Norton1952}
and Laywine--Mullin~\cite[Thm.~6.6]{laywine1998discrete},
though we strengthen the proposition by showing that when the squares are Latin (not just column-Latin)
the square providing the guarantee of orthogonality arises as a transversal representation
of one of the original two squares.

\begin{proposition}\label{orthogonal pair result}
Let\/ $C$ and $F$ be column-Latin squares. Then $(C,F)$ is an orthogonal pair if and only if there is a Latin square $Z$ such that $ZC = F$.
Moreover, if in addition, $C$ is a Latin square, then $(Z,F)$ is a TRP\@.
\end{proposition}

\begin{proof}
Suppose $Z$ is a Latin square and $ZC=F$ for column-Latin squares $C$ and $F$.
By Lemma~\ref{A and E}, $(Z,E)$ is an orthogonal pair.
By Lemma~\ref{orthogonal product}, $(ZC,EC)$ is an orthogonal pair. 
Since $ZC=F$ and $EC=C$, it follows that $(F,C)$ is an orthogonal pair.

Conversely, suppose $(C,F)$ is an orthogonal pair.
Let $Z = FC^{-1}$ (i.e., $ZC=F$).
Since $(C,F)$ is an orthogonal pair, by Lemma~\ref{orthogonal product}, $(Z,E)$ is an orthogonal pair (since $FC^{-1} = Z$ and $CC^{-1} = E$).
By Lemma~\ref{A and E}, $Z$ is a Latin square.

We now show that if $C$ is a Latin square and $F$ is a column-Latin square 
such that $(C,F)$ is an orthogonal pair,
then $(Z,F)$, which is equal to $(Z,ZC)$, is a TRP\@. 
Suppose that $(Z,F)$ is not a TRP\@.
Then there exist $i$, $i'$, $j$, $j' \in \{0,1,2,\dotsc,n-1\}$ where $j\neq j'$ with
\begin{align*}
Z[i,j] &= ZC[i',j] = Z[C[i',j],j] , \text{ and } \\
Z[i,j'] &= ZC[i',j'] = Z[C[i',j'], j'] .
\end{align*}
Since $Z$ is a Latin square, the symbols in each of its columns are distinct.
Thus, considering the entries of column~$j$ of $Z$,
we must have $C[i',j] = i$ and $C[i',j']=i$, but $C[i',j]=C[i',j']$ is a contradiction because
the rows of $C$ (in particular, row~$i'$) are permutations, implying $j=j'$.
Thus $(Z,F)$ is a TRP\@. \qed
\end{proof}

\subsection{Orthogonal Pair / Transversal Representation Duality}\label{sec:op_tr}

We now state a duality between orthogonality and transversal representations.
This duality was already used by Myrvold~\cite[Thm~1.1]{myrvold1999negative},
but we show how the duality can be concisely formulated in terms
of the composition operation on column-Latin squares---a convenient viewpoint
that we were unable to find in the literature.
Roughly speaking, Lemmas~\ref{R and E} and~\ref{product}
are the analogue of Lemmas~\ref{A and E} and~\ref{orthogonal product}
with ``orthogonal pair'' replaced by ``transversal representation pair''.

\begin{lemma}\label{R and E}
Let\/ $C$ be a column-Latin square. Then $(C,E)$ is a TRP if and only if\/ $C$ is a Latin square.
\end{lemma}

\begin{proof}
Let $C$ be a column-Latin square and $(C,E)$ be a TRP\@. It is enough to show that rows of $C$ are each an $n$-permutation. Assume, for a contradiction, that this is not the case. Then for some $0\leq i,j,j',k <n$ with $j \neq j'$, $C[i,j] =k= C[i,j']$. Since
$E$ is a transversal representation of $C$,
row $i$ of $C$ has its $t$-th symbol from column $t$ of $E$. Therefore, the symbol $k$ is on two different rows of $E$, which contradicts the definition of $E$. Therefore, rows of $C$ are each an $n$-permutation, and consequently, $C$ is a Latin square. 

Conversely, suppose $C$ is a Latin square. Since all symbols are distinct on each row of $C$ and the same on each row of $E$, then each row of $C$ takes symbols from distinct rows and columns of $E$ and the $t$-th symbol on each row is from column $t$ of $E$. Thus $E$  is a transversal representation of $C$. It follows that $(C,E)$ is a TRP\@. \qed
\end{proof}

\begin{lemma}\label{product}
Let\/ $\{C_1, C_2, \dotsc, C_m\}$ be a set of mutual TRPs of column-Latin squares, then for any column-Latin square $G$, the set $\{GC_1, GC_2, \dotsc, GC_m\}$ comprises mutual TRPs. 
\end{lemma}

\begin{proof}
It is enough to prove this statement for a set of two column-Latin squares. The columns of $GC_1$ and $GC_2$ are compositions of two permutations, therefore $GC_1$ and $GC_2$ are column-Latin squares. 
Assume, for a contradiction, that this is not the case.
Suppose there exist $i$, $i'$, $j$, $j' \in \{0,1,2,\dotsc,n-1\}$ where $j\neq j'$ with
\[ GC_1[i, j] = GC_2[i', j] \text{ and } GC_1[i, j'] = GC_2[i', j']. \]
Thus by equality of the symbols 
\[ G[C_1[i,j],j] = G[C_2[i',j],j] 
\text{ and }
G[C_1[i,j'],j'] = G[C_2[i',j'],j']. \]
Since $G$ is a column-Latin square, the uniqueness of symbols in its columns provides that
\[ C_1[i,j] = C_2[i',j] \text{ and } C_1[i,j'] = C_2[i',j'].\]
Since $(C_1,C_2)$ is a TRP, we have $j=j'$. This contradicts our assumption. Thus $(GC_1,GC_2)$ is a TRP\@. Therefore, the set consists of mutual TRPs. \qed
\end{proof}

\begin{proposition}\label{transversal pair result}
Let $C$ and $F$ be column-Latin squares.
Then $(C,F)$ is a TRP if and only if there is a Latin square $Z$ such that $CZ = F$.
Moreover, if\/ $C$ is a Latin square, then $Z$ is orthogonal to $F$.
\end{proposition}

\begin{proof}
Assume there exists a Latin square $Z$ such that $CZ = F$.
By Lemma~\ref{R and E}, $(Z,E)$ is a TRP\@.
By Lemma~\ref{product}, $(C,F)$, which is equal to $(CE,CZ)$, is a TRP\@.

Conversely, assume $(C,F)$ is a TRP\@. Let $Z = C^{-1}F$.
Since $(C,F)$ is a TRP and $(C^{-1}C ,C^{-1}F)=(E,Z)$, by Lemma~\ref{product}, $(E,Z)$ is a TRP\@.
Thus $(E,Z)$ is a TRP\@.
We have that $Z$ is a Latin square by Lemma~\ref{R and E}.

Now we prove that if $C$ is a Latin square, $Z$ and $F$ are orthogonal.
Assume, for a contradiction, that $(Z,F)$ (where $F=CZ$) is not an orthogonal pair, i.e.,
there exist $i$, $i'$, $j$, $j'\in \{0,1,2,\dotsc,n-1\}$ with $j\neq j'$ for which
\[ Z[i,j] = Z[i',j'] \text{ and } F[i,j] = F[i',j'] . \]
The second equation implies $C[Z[i,j],j] = C[Z[i',j'],j']$ an equality between two symbols in rows $j$ and~$j'$ of $C$,
which, after using the first equation, yields
$C[Z[i,j],j] = C[Z[i,j],j']$.
Since $C$ is a Latin square, its rows are permutations, which implies $j=j'$
and contradicts the assumption that $j\neq j'$.
Therefore, $(Z,F)$ must be an orthogonal pair. \qed
\end{proof}
    
The following result describes the equivalence between
a set of mutually orthogonal column-Latin squares
and a set of mutually TRPs.
The correctness of our SAT encoding relies on this equivalence.

\begin{theorem}[cf.~\cite{myrvold1999negative}]\label{ortho-to-transv}
Let $\mathcal{C}$ denote a set $\{C_1,\dotsc,C_r\}$ of\/ $r$ column-Latin squares of order~$n$.
\begin{enumerate}
\item[(a)] If\/ $\mathcal{C}$ contains
mutually orthogonal squares, then the set
\[ \{\, Z_1, \dotsc, Z_r : Z_1= C_1, Z_t = C_1C_t^{-1} \text{ for } 2\leq t \leq r \,\} \]
contains mutual TRPs.
\item[(b)] If\/ $\mathcal{C}$ consists of mutual TRPs, then the set
\[ \{\, Y_1, \dotsc, Y_r : Y_1= C_1, Y_t=C_t^{-1}C_1 \text{ for } 2\leq t \leq r \,\} \] 
contains mutually orthogonal pairs.  
\end{enumerate}
\end{theorem}

\begin{proof}
For (a), suppose the set $\{\,C_i : 1\leq i \leq r\,\}$ consists of mutually orthogonal column-Latin squares of order $n$. Construct a set of $r$ squares $\{\,Z_i : 1\leq i \leq r\,\}$ by letting $Z_1 = C_1$ and $Z_t =C_1C_t^{-1}$ for $2\leq t \leq r$. Proposition~\ref{orthogonal pair result} gives that each $Z_t$, $2\leq t \leq r$  is a Latin square; further it ensures that $(Z_1,Z_t)$ is a TRP\@. Observe that $Z_tC_tC_s^{-1} = Z_s$ for $2\leq t,s \leq r$ where $t\neq s$. Since both $C_t$ and $C_s^{-1}$ are column-Latin squares, their composition is a column-Latin square.  Thus $(Z_t,Z_s)$ for $2\leq t,s \leq r$ where $t\neq s$, being a TRP also follows from Proposition~\ref{orthogonal pair result}. 

For (b), suppose the set $\{\,C_i: 1\leq i \leq r\,\}$ consists of column-Latin squares of order~$n$ such that any two squares form a TRP\@. Construct a set of $r$ squares $\{\,Y_i : 1\leq i \leq r\,\}$ by letting $Y_1 = C_1$ and $Y_t=C_t^{-1}C_1$ for $2\leq t \leq r$. Proposition~\ref{transversal pair result} gives that each $Y_t$, $2\leq t \leq r$  is a Latin square; and that $Y_1$ and~$Y_t$ are orthogonal. Observe that $C_s^{-1}C_tY_t = Y_s$ for $2\leq t,s \leq r$ where $t\neq s$. Since both $C_s^{-1}$ and $C_t$ are column-Latin squares, their composition is a column-Latin square. Therefore, $Y_t$ being orthogonal to~$Y_s$ for $2\leq t,s \leq r$ where $t\neq s$ also follows from Proposition~\ref{transversal pair result}. \qed
\end{proof}

\section{Encoding and Implementation}\label{sec:encoding}

In this section we describe our encoding of the problem of constructing transversal representation pairs (TRPs)
into a Boolean satisfiability problem and how we use our encoding to search for TRPs for each of Myrvold's 28
possible types described in Section~\ref{sec:trans_rep_types}.
Recall that Myrvold's 28 types describe TRPs $(P,Q)$ for which $P$ and $Q$ are
each transversal representations of a Latin square $L$ of order $n=10$
containing a $4\times4$ Latin subsquare.

To reduce the existence of the $n\times n$ square $P$ into Boolean logic,
we use $n^3$ Boolean variables $P_{i,j,k}$ (for $0\leq i,j,k < n$)
with $P_{i,j,k}$ denoting the fact that the $(i,j)$th entry of $P$ is $k$.
Similarly, another $n^3$ Boolean variables $Q_{i,j,k}$ for $0\leq i,j,k < n$
represent the entries of the square $Q$.

Once these variables have been defined, we need to specify constraints that $P$
and~$Q$ are Latin squares (see Section~\ref{sec:symbol}),
are a transversal representation pair (see Section~\ref{sec:TRPC}),
and conform to one of Myrvold's 28 types (see Section~\ref{sec:colour}).
Additionally, we ensure that the white entries in the last four columns of $P$ and~$Q$
appear in a way that is consistent with a $4\times4$ Latin subsquare $\Omega$
being in a square $L$ having mutual transversal representations
$P$ and~$Q$
(see Section~\ref{sec:subsquare_consistency}).
We also describe a method of symmetry breaking
which reduces the size of the search space by adding additional constraints
which hold without loss of generality (see Section~\ref{sec:symmetry_breaking}).
Finally, once we have found a collection of TRPs, we run a postprocessing step on
them, ensuring that the TRPs are pairwise inequivalent and that they
cannot be extended to a set of three mutual TRPs (see Section~\ref{sec:extendability}).
Our encoding scripts are written in Python and are freely available
at \href{https://doi.org/10.5281/zenodo.18130631}{doi.org/10.5281/zenodo.18130631}.

\subsection{Latin Square Constraints}\label{sec:symbol}

First, we need to describe constraints on the
variables $P_{i,j,k}$ (meaning that $P[i,j]=k$) asserting that $P$
is a Latin square.  Direct methods for doing this from the definition of a
Latin square are well known and widely used; e.g., see (10.1)--(10.4) in Zhang's survey~\cite{zhang1997specifying}.
The direct method asserts that every cell of $P$ contains \emph{at least one} symbol
and \emph{at most one} symbol, i.e.,
\[ \bigvee_{0\leq i<n} P_{p,q,i} \quad\text{and}\quad \bigwedge_{0\leq i < j < n} \left(\neg P_{p,q,i} \lor \neg P_{p,q,j} \right) \quad\text{for all $0\leq p,q<n$.} \]
Additionally, every column of $P$ contains $n$ distinct symbols,
\[ \bigvee_{0\leq i<n} P_{i,q,r} \quad\text{and}\quad \bigwedge_{0\leq i < j < n} \left(\neg P_{i,q,r} \lor \neg P_{j,q,r} \right) \quad\text{for all $0\leq q,r<n$,} \]
and similarly every row of $P$ contains $n$ distinct symbols,
\[ \bigvee_{0\leq i<n} P_{p,i,r} \quad\text{and}\quad \bigwedge_{0\leq i < j < n} \left(\neg P_{p,i,r} \lor \neg P_{p,j,r} \right) \quad\text{for all $0\leq p,r<n$.} \]
This encoding uses what is known as
the binomial or pairwise encoding of the \emph{exactly one} predicate~\cite{MarquesSilva}
and uses $3n^2\bigl(\binom{n}{2}+1\bigr)$ clauses in total.
While this encoding gave good performance, in our experiments we got slightly better performance
with the cardinality constraint encoding of Bailleux and Boufkhad~\cite{Bailleux2003}.
Their encoding reduces a constraint like $x_1+\dotsb+x_n=r$ (where $r$ is a fixed integer
between~$0$ and~$n$ and we think of the Boolean $x_i$s as $\{0,1\}$ variables)
into conjunctive normal form.  Using this encoding we specify that $P$ is a Latin square with the cardinality constraints
\[ \sum_{0\leq i<n} P_{p,q,i} = 1, \quad \sum_{0\leq i<n} P_{i,p,q} = 1, \quad \sum_{0\leq i<n} P_{p,i,q} = 1 \quad\text{for all $0\leq p,q<n$,} \]
and a similar encoding can be used to specify that $Q$ is also a Latin square.

\subsection{Transversal Representation Constraints}\label{sec:TRPC}

The direct encoding that $(P,Q)$ is a TRP using the contrapositive of Definition~\ref{def:transrep}
would be
\[ (P_{i,j,k}\land P_{i,j',k'}\land Q_{i',j,k})\limp \lnot Q_{i',j',k'} \quad\text{for all $0\leq i,i',j,j',k,k'<n$ with $j<j'$}. \]
This is because if row $i$ of $P$ has its $j$th entry as $k$ and its $(j')$th entry as $k'$,
then in whatever row of~$Q$ which has its $j$th entry as $k$ (one such row must exist since $Q$ is a Latin square)
that row \emph{cannot} have its $(j')$th entry as $k'$, or that row wouldn't represent a transversal.
However, this encoding uses $n^4\binom{n}{2}=\Theta(n^6)$ clauses of length 4 which is not ideal in practice.
Instead, our encoding that $(P,Q)$ is a TRP will assert the existence
of the Latin square $Z=P^{-1}Q$ and by Proposition~\ref{transversal pair result} this
implies that $P$ and~$Q$ are a transversal representation pair.

As before, the entries of the square $Z$ are encoded via $n^3$
new variables $Z_{i,j,k}$ (with $0\leq i,j,k<n$) and $Z$ is enforced
to be a Latin square using the same encoding described in Section~\ref{sec:symbol}.
Now we need to enforce the relationship $Q=PZ$, which means
that the $(i,j)$th entry of $Q$ is equal to the $(i',j)$th entry of~$P$, where $i'=Z[i,j]$.
Letting $k$ represent the $(i,j)$th entry of $Q$, this gives the constraints
\[ (Z_{i,j,i'}\land P_{i',j,k})\limp Q_{i,j,k} \quad\text{for all $0\leq i,i',j,k<n$.} \]
Moreover, because $P=QZ^{-1}$ and $Z=P^{-1}Q$, we similarly derive the constraints
\begin{gather*}
(Z_{i,j,i'}\land Q_{i,j,k})\limp P_{i',j,k} \quad\text{for all $0\leq i,i',j,k<n$,} \\
(P_{i',j,k}\land Q_{i,j,k})\limp Z_{i,j,i'} \quad\text{for all $0\leq i,i',j,k<n$.}
\end{gather*}
These last two kinds of constraints are technically redundant, but we found that
they tended to improve the performance
of the solving in practice.

Thus, our encoding that $(P,Q)$ is a TRP uses $3n^4$ clauses and the
$3n^2$ cardinality constraints
$\sum_{i}Z_{i,j,k}=\sum_{i}Z_{j,k,i}=\sum_{i}Z_{j,i,k}=1$ for all $0\leq j,k<n$.
Altogether, this TRP encoding
uses $\Theta(n^4)$ clauses of length at most 3, and in practice this
is preferable to the $\Theta(n^6)$ clauses of length 4 used by the direct
encoding.

A similar $\Theta(n^4)$ clause encoding was previously derived
by Zhang (see~\cite[Lemma~2]{zhang1997specifying}), for ensuring
the orthogonality of a pair $(A,B)$ of Latin squares of order~$n$.  Zhang's encoding
for orthogonality uses a new predicate $\Phi(i,j,k)$ introduced via a clever
trick and Zhang mentions that \emph{``It is a
challenge to develop a method which can automatically generate the predicates
like $\Phi$\dots''}~\cite{Zhang2021}.
Zhang does not view $\Phi$ as a square, but
viewing $\Phi(i,j,k)$ as asserting that $\Phi[i,j]=k$,
Zhang uses constraints saying that
$\Phi$'s columns have distinct symbols and that
the entries of $A$ and $B$ determine $\Phi$'s entries.
Following our notation, Zhang uses constraints of the form
\[ (A_{i,j,k}\land B_{i,j,\ell})\limp\Phi(i,k,\ell) , \qquad\text{for all $0\leq i,j,k,\ell<n$} . \]
In light of the above and Proposition~\ref{orthogonal pair result}, this means that
not only is $\Phi$ itself a Latin square, it
can be naturally viewed as a transversal representation of
one of the original Latin squares and
conveniently expressed via a composition square.\footnote{The
constraints used by Zhang causes the \emph{columns} of $\Phi$
to represent transversals of $B$ and for $\Phi$
to be the composition square $BA^{-1}$
where the composition and inverse are defined
\emph{row-wise} instead of column-wise like in the rest of this paper.}
Viewing $\Phi$ as a composition square, one can derive
additional constraints on $\Phi$ using this extra structure (e.g., the entries of $A$ and~$\Phi$
determine the entries of~$B$).  As previously mentioned, such constraints are technically
redundant, but tended to help the efficiency of the solver
in our experiments.

\subsection{Colour Constraints}\label{sec:colour}

We now describe how we encode that the square $P$ is one
of Myrvold's eight types described in Table~\ref{tbl:transtypes};
an identical encoding is used for $Q$.
In order to do this, we need to be able to specify the \emph{colour}
of each cell in the square $P$ to be either white, light, or dark.
Let $\w$ and $\d$ represent fixed symbols that are not
in our symbol set $\{0,\dotsc,n-1\}$.

We let the Boolean variable $P_{i,j,\w}$ represent that the $(i,j)$th entry of $P$ is white,
and let the Boolean variable $P_{i,j,\d}$ represent that the $(i,j)$th entry of $P$ is dark.
Otherwise,
if both $P_{i,j,\w}$ and $P_{i,j,\d}$ are false, then the $(i,j)$th entry
of $P$ will be light.  Note that dark variables are only necessary in the first
six columns, since no dark entries appear in the last four columns (see
Figure~\ref{fig:decomp}).  Additionally, the position of the dark cells in the
first six columns completely determines the position of the white cells in the
first six columns---the whites containing the symbols $\{4,\dotsc,9\}$
not darkly coloured---making the variables $P_{i,j,\w}$ only necessary for $j\geq6$.
Altogether, we introduce $n^2$ new variables encoding the colours of $P$.

To ensure the symbols $\{0,\dotsc,3\}$ are coloured white, we use the clauses
\[ P_{i,j,r} \limp P_{i,j,\w} \quad\text{for all $0\leq i<n$, $6\leq j<n$, and $0\leq r<4$,} \]
and conversely to ensure that only symbols $\{0,\dotsc,3\}$ are coloured white we use
$P_{i,j,\w} \limp \bigvee_{0\leq r<4} P_{i,j,r}$ for all $0\leq i<n$ and $6\leq j<n$.
Similarly, to ensure that only symbols $\{4,\dotsc,9\}$ are coloured dark, we 
use the clauses
\[ P_{i,j,\d} \limp \bigvee_{4\leq r<n} P_{i,j,r} \quad\text{for all $0\leq i<n$ and $0\leq j<6$.} \]
Recall that a transversal is said to be of type $p_k$ when it has $k$ whites
in its last four entries.  By Lemma~\ref{lem:darkcount}, transversals of type~$p_k$ will also
have $2k-2$ dark entries in its first six entries.  Thus, in order to specify
that row $i$ in $P$ is of type~$p_k$, we use the constraints
\[ \sum_{0\leq j<6}P_{i,j,\d} = 2k-2 \quad\text{and}\quad \sum_{6\leq j<n}P_{i,j,\w} = k . \]
Here, like in Section~\ref{sec:symbol}, we think of Boolean variables as taking
$\{0,1\}$ values and encode the cardinality constraints with the
encoding of Bailleux and Boufkhad~\cite{Bailleux2003}.  We also know
that each of the first six columns of $P$ contain exactly two dark entries,
so we use the cardinality constraints
\[ \sum_{0\leq i<n}P_{i,j,\d} = 2 \quad\text{for all $0\leq j<6$.} \]

Similarly, we also use $n^2$ Boolean variables $Q_{i,j,\w}$ and $Q_{i,j,\d}$
to represent the colours of the square $Q$ and add similar constraints to
those above (using the $Q_{i,j,\w}$ and $Q_{i,j,\d}$ variables
in place of the $P_{i,j,\w}$ and $P_{i,j,\d}$ variables).
We now have specified a coloured TRP $(P,Q)$ with each of $P$ and $Q$ conforming
to any of Myrvold's types $\R$, $\S$, $\dotsc$, $\X$ selected in advance.
However, because $P$ and $Q$ are
both transversal representations of the same coloured square $L$, it is important
that their colours be \emph{consistent} between themselves.  In particular, the 
two entries coloured dark in each of the first six columns of $P$ must match
the two entries coloured dark in each of the first six columns of $Q$.  (The
white colours always match as they correspond exactly to the symbols $\{0,1,2,3\}$, so
if the dark colours match then so must the light colours.)

Suppose the $(i,j)$th entry of $P$ has symbol $k$ and is coloured dark.
Then, in order for the colouring to be consistent, the
entry of $Q$ in the $j$th column having symbol $k$ must also be coloured dark.
The symbol~$k$ must exist in the $j$th column of $Q$
because $Q$ is a Latin square, so say this happens in row $i'$.
Then to express the consistency of the colours in $P$ and $Q$ we use the constraints
\[ (P_{i,j,k}\land P_{i,j,\d}\land Q_{i',j,k}) \limp Q_{i',j,\d} \quad\text{for all $0\leq i,i'<n$, $0\leq j<6$, and $4\leq k<n$.} \]
Although not strictly necessary, we also add constraints
deriving the colour of cell $(i,j)$ in $P$ from the colour of cell $(i',j)$ in $Q$,
giving the constraints
\[ (P_{i,j,k}\land Q_{i',j,\d}\land Q_{i',j,k}) \limp P_{i,j,\d} \quad\text{for all $0\leq i,i'<n$, $0\leq j<6$, and $4\leq k<n$.} \]

\subsection{Consistency with the $4\times4$ Subsquare $\Omega$}\label{sec:subsquare_consistency}

Recall Myrvold's seven transversal representation types of a Latin square $L$
are under the assumption that $L$ has a $4\times4$ Latin subsquare $\Omega$.
As described in Section~\ref{sec:trans_rep_types}, we assume that the subsquare
$\Omega$ contains the symbols $\{0,1,2,3\}$ and appears in the lower-right of $L$.
There are two possibilities for~$\Omega$ up to isotopism, where two Latin squares
are \emph{isotopic} if one can be transformed into the other by row, column,
or symbol permutations~\cite{mckay2007small}.  The two possibilities 
for $\Omega$ up to isotopism are the Cayley tables
of~$\Z_4$ and $\Z_2\times\Z_2$,
and we assume
that $\Omega$ is either
\[ \Omega_1 \coloneqq \mspace{3mu}
\begin{tabular}{| *4{ @{\hskip 0pt} c @{\hskip 0pt} |}}
\hline
\sq0 & \sq1 & \sq2 & \sq3 \\ \hline
\sq1 & \sq2 & \sq3 & \sq0 \\ \hline
\sq2 & \sq3 & \sq0 & \sq1 \\ \hline
\sq3 & \sq0 & \sq1 & \sq2 \\ \hline
\end{tabular}\mspace{6mu},
\qquad\text{or}\qquad
\Omega_2 \coloneqq \mspace{3mu}\begin{tabular}{| *4{ @{\hskip 0pt} c @{\hskip 0pt} |}}
\hline
\sq0 & \sq1 & \sq2 & \sq3 \\ \hline
\sq1 & \sq0 & \sq3 & \sq2 \\ \hline
\sq2 & \sq3 & \sq0 & \sq1 \\ \hline
\sq3 & \sq2 & \sq1 & \sq0 \\ \hline
\end{tabular}\mspace{6mu}. \]
Since we are searching for Latin squares $P$ and $Q$ that are
both transversal representations of $L$, this restricts the
possible locations for the white entries in the last four
columns of $P$ and $Q$.  For example, if either $\Omega_1$
or $\Omega_2$ is the lower-right subsquare of $L$, then since
$P$ is a transversal representation of $L$, it cannot be
the case that $P[i,6]=0$ and $P[i,7]=1$,
regardless of the row~$i$ chosen.
This is because the $0$ in column~6 of $L$ and the $1$ in column~7 of $L$
appear in the same row and therefore cannot appear in the same transversal.

Noting that the first row of
$\Omega_1$ and $\Omega_2$ are both $[0,1,2,3]$, we add the clauses
\[ P_{i,j,j-6}\limp\lnot P_{i,j',j'-6} \quad\text{for all $0\leq i<n$ and $6\leq j<j'<n$}, \]
and use similar clauses for $Q$.  Generalizing this, let $\omega_1$ be a Boolean variable that is
true when $\Omega_1$ is to be used in $L$, and let $\omega_2$ be a Boolean variable
that is to be true when $\Omega_2$ is to be used in $L$.  We add the clauses
\begin{gather*}
(\omega_1\land P_{i,j,\Omega_1[i',j-6]})\limp\lnot P_{i,j',\Omega_1[i',j'-6]} \\
(\omega_2\land P_{i,j,\Omega_2[i',j-6]})\limp\lnot P_{i,j',\Omega_2[i',j'-6]}
\end{gather*}
for all $0\leq i<n$, $i'\in\{1,2,3\}$, and $6\leq j<j'<n$, and use similar clauses for~$Q$.
Specifying either $\Omega_1$ or~$\Omega_2$ is to be used in $L$
is done with the clause $\omega_1\lor\omega_2$.
If a particular subsquare $\Omega_1$ or $\Omega_2$ is desired,
it can be enforced with either the unit clause $\omega_1$ or
the unit clause $\omega_2$.

\subsection{Symmetry Breaking}\label{sec:symmetry_breaking}

The ordering of rows of a transversal representation square
is arbitrary
in the sense that if $P$ is a transversal representation
of $Q$, then the rows of $P$ can be freely permuted while preserving
the fact that it is a transversal representation of~$Q$.
Similarly, the rows of $Q$ may also be permuted.
Columns may not be permuted independently, but if $(P,Q)$ is a TRP
and the same permutation of columns is applied to both
$P$ and $Q$ simultaneously, then the resulting new pair will also be a TRP\@.
Similarly, the same permutation of symbols applied to both squares in a TRP
maintains the property of the pair being a TRP\@.
Since we have already supposed that the symbols in the lower-right $4\times4$
submatrix of~$L$ are in $\{0,1,2,3\}$, in order to not disturb this structure
all permutations on symbols will operate on $\{0,1,2,3\}$ and $\{4,\dotsc,9\}$
independently.  Similarly, we only use permutations of the first six and last
four columns when
transforming a TRP into the normal form
defined below.

By a \emph{coloured} TRP we mean one whose cells have been assigned
the colours $\{\text{white},\text{light},\text{dark}\}$
corresponding to Myrvold's types from Section~\ref{sec:trans_rep_types}.
If $(P,L)$ is a coloured TRP where~$L$ has been coloured
corresponding to Figure~\ref{fig:decomp}, then permutations of the
rows of $P$ will also permute the colour positions in~$P$.  Similarly, permutations
of the columns of $P$ and $L$ simultaneously will permute
the colour positions in $(P,L)$, whereas
permuting the symbols $\{4,\dotsc,9\}$ or $\{0,1,2,3\}$
in $(P,L)$ will not permute the colour positions in $(P,L)$.

Row permutations of $P$, row
permutations of $Q$, column permutations of the first six or last four columns
of $(P,Q)$,
and symbol permutations of the first four or last six symbols of $(P,Q)$
generate a group $G$ of size $10!^2\cdot6!^2\cdot4!^2\approx4\cdot10^{21}$.
We call two coloured TRPs \emph{equivalent} if one can be transformed to the
other using operations in $G$.
The large size of $G$ means that our search space contains a large number of TRPs
that are equivalent.
This artificially increases the size of the search space,
and we would like to constrain
the search space in order to limit
the search to as few representatives from each equivalence class
as possible---this is known as \emph{symmetry breaking}.
We are able to remove many representatives from the search by
only searching for TRPs in the normal form defined below.

\begin{definition}\label{normal form}
    A coloured TRP $(P,Q)$ is in \emph{normal form} if
    the rows of each square are sorted by transversal type
    (i.e., if row $i$ has type $p_k$ and row $i'\geq i$ has type $p_{k'}$
    then $k\leq k'$), all the rows of the same transversal type are sorted
    in increasing lexicographic order, and the first row of\/ $P$ is one of
    \begin{center}
    \begin{tikzpicture}[baseline={([yshift=-.5ex]current bounding box.center)}]
\matrix (m) [matrix of nodes,nodes={element},column sep=-\pgflinewidth, row sep=-\pgflinewidth]{
|[draw,fill=white]| $0$ &
|[draw,fill=white]| $1$ &
|[draw,fill=white]| $2$ &
|[draw,fill=light]| $4$ &
|[draw,fill=light]| $5$ &
|[draw,fill=light]| $6$ &
|[draw,fill=white]| $3$ &
|[draw,fill=light]| $7$ &
|[draw,fill=light]| $8$ &
|[draw,fill=light]| $9$ \\
}; \end{tikzpicture}\rlap{\!,} \\
    \begin{tikzpicture}[baseline={([yshift=-.5ex]current bounding box.center)}]
\matrix (m) [matrix of nodes,nodes={element},column sep=-\pgflinewidth, row sep=-\pgflinewidth]{
|[draw,fill=white]| $0$ &
|[draw,fill=white]| $1$ &
|[draw,fill=white]| $3$ &
|[draw,fill=light]| $4$ &
|[draw,fill=light]| $5$ &
|[draw,fill=light]| $6$ &
|[draw,fill=white]| $2$ &
|[draw,fill=light]| $7$ &
|[draw,fill=light]| $8$ &
|[draw,fill=light]| $9$ \\
}; \end{tikzpicture}\rlap{\!, or} \\
    \begin{tikzpicture}[baseline={([yshift=-.5ex]current bounding box.center)}]
\matrix (m) [matrix of nodes,nodes={element},column sep=-\pgflinewidth, row sep=-\pgflinewidth]{
|[draw,fill=white]| $0$ &
|[draw,fill=white]| $2$ &
|[draw,fill=white]| $3$ &
|[draw,fill=light]| $4$ &
|[draw,fill=light]| $5$ &
|[draw,fill=light]| $6$ &
|[draw,fill=white]| $1$ &
|[draw,fill=light]| $7$ &
|[draw,fill=light]| $8$ &
|[draw,fill=light]| $9$ \\
}; \end{tikzpicture}\rlap{\!.}
    \end{center}
\end{definition}

In Theorem~\ref{isomorphism pro}, we demonstrate that
every equivalence class of TRPs of the kind we are looking for
contains at least one TRP in normal form.  First, we prove
a simple lemma used in the proof of Theorem~\ref{isomorphism pro}.

\begin{lemma}\label{lem:4x4}
Suppose $\Omega$ is a Latin square of order 4.
Then $\Omega$ is isotopic to
either $\Omega_1$ or\/~$\Omega_2$.  In either case,
$\Omega$ can be transformed into $\Omega_1$ or $\Omega_2$
without permuting column~0 or symbol~0.
\end{lemma}

\begin{proof}
There are exactly two
Latin squares of order~4 up to isotopy ($\Omega_1$ and~$\Omega_2$)
and a total of four \emph{reduced}
Latin squares of order 4 (i.e., with entries in the first
row and column appearing in sorted order)~\cite{mckay2007small}.  The two additional reduced
Latin squares of order four are both isotopic to $\Omega_1$
and are given by
\[ \Omega_3 \coloneqq \mspace{3mu}
\begin{tabular}{| *4{ @{\hskip 0pt} c @{\hskip 0pt} |}}
\hline
\sq0 & \sq1 & \sq2 & \sq3 \\ \hline
\sq1 & \sq0 & \sq3 & \sq2 \\ \hline
\sq2 & \sq3 & \sq1 & \sq0 \\ \hline
\sq3 & \sq2 & \sq0 & \sq1 \\ \hline
\end{tabular}\mspace{6mu},
\qquad\text{and}\qquad
\Omega_4 \coloneqq \mspace{3mu}\begin{tabular}{| *4{ @{\hskip 0pt} c @{\hskip 0pt} |}}
\hline
\sq0 & \sq1 & \sq2 & \sq3 \\ \hline
\sq1 & \sq3 & \sq0 & \sq2 \\ \hline
\sq2 & \sq0 & \sq3 & \sq1 \\ \hline
\sq3 & \sq2 & \sq1 & \sq0 \\ \hline
\end{tabular}\mspace{6mu}. \]

If $\Omega$ is isotopic to $\Omega_2$, it can be transformed into reduced form
by using row permutations to put $0$ in the upper-left corner,
then symbol permutations of $\{1,2,3\}$ to make the first row $[0,1,2,3]$,
and then permuting the last three rows to transform the first column into $[0,1,2,3]$.
Since there is only one reduced Latin square of order~4 isotopic to $\Omega_2$,
this must transform $\Omega$ to $\Omega_2$.

Otherwise, if $\Omega$ is isotopic to $\Omega_1$, use row and symbol permutations
as above to transform it into reduced form, thereby transforming it into $\Omega_1$,
$\Omega_3$, or $\Omega_4$.  To transform $\Omega_3$ into $\Omega_1$, swap
columns~1 and~2, rows~1 and~2, and symbols~1 and~2.  To transform $\Omega_4$
into $\Omega_1$, swap columns~2 and~3, rows~2 and~3, and symbols~2 and~3. \qed
\end{proof}

\begin{theorem}\label{isomorphism pro}
Suppose $(P,Q)$, $(P,L)$, and $(Q,L)$ are coloured TRPs
where $L$ contains a $4\times4$ Latin subsquare
and is coloured according to Figure~\ref{fig:decomp}.
Then $(P,Q)$ is equivalent to a coloured TRP in normal form
and the lower-right\/ $4\times4$ Latin subsquare in $L$ can be taken to be
either\/ $\Omega_1$ or\/ $\Omega_2$.
\end{theorem}

\begin{proof}
Let $(P,Q)$ be a coloured TRP satisfying the preconditions
of the theorem that we want to transform to a pair in normal form.
First, permute the rows of $P$ to put together rows
of the same transversal type~$p_i$ (for $i\in \{1,2,3,4\}$)
such that all rows of type $p_k$ come before
all rows of type $p_{k'}$ when $k<k'$.  Next, permute the rows of $Q$
in a similar fashion so the rows of $Q$ are also sorted by transversal type.

Since all square types contain transversals of type $p_1$,
and none contain transversals of type $p_0$, the
above sorting process implies the first row of $P$ is of type $p_1$.
Now use column permutations of the first six columns (and the last
four columns) to position the colours of the first row of $P$
in the following order: 3~white, 3 light, 1 white, 3 light.
Following this, if the symbol of $P$ in the upper-left corner
is not symbol~0, use a symbol permutation to make it 0.

By Lemma~\ref{lem:4x4},
we can now use symbol permutations of $\{1,2,3\}$,
simultaneous permutations of the last three columns of $(P,Q,L)$, and row permutations in~$L$
to ensure that the lower-right $4\times4$ subsquare of $L$ is either $\Omega_1$
or $\Omega_2$.
Row permutations of $L$, symbol permutations of $\{1,2,3\}$, and column
permutations of the last three columns will not disturb
the colouring of the first row of $P$ or the fact $P[0,0]=0$.

Afterward, apply permutations of the symbols $\{4,\dotsc,9\}$
to $P$, $Q$, and $L$ simultaneously to put the light
entries of the first row of $P$ into normal form.  If $P[0,1]$
and $P[0,2]$ are not in ascending order, use a column permutation
to sort them.  As a result, the first three entries of the
first row of $P$ are now $[0,1,2]$, $[0,1,3]$, or $[0,2,3]$,
so the first row of $P$ is in one of the three cases
given in Definition~\ref{normal form}.

Finally, within each subset of rows of the same transversal type of $P$ (and
independently $Q$),
permute the rows so they appear in increasing lexicographic order.
The first row of $P$ already begins with the symbol~$0$,
so it will not be moved. \qed
\end{proof}

Thus, without loss of generality we can assume the TRP we are searching
for is in normal form and so we add extra constraints into our encoding
to enforce this.
Fixing the lightly coloured entries in the first row of $P$
and the $(0,0)$th symbol of $P$
can be done by adding appropriate unit clauses (clauses of length~1),
namely,
\[ P_{0,0,0}\land\bigwedge_{3\leq j\leq 5}P_{0,j,j+1}\land\bigwedge_{7\leq j\leq 9}P_{0,j,j} . \]
The remaining entries in the first row of $P$ are determined by
the value of $P[0,6]$,
giving the constraints
\[ P_{0,6,3}\limp (P_{0,1,1}\land P_{0,2,2}), P_{0,6,2}\limp (P_{0,1,1}\land P_{0,2,3}), \text{ and } P_{0,6,1}\limp (P_{0,1,2}\land P_{0,2,3}) . \]
Each constraint $x\limp (y\land z)$ is broken into two clauses of length two ($x\limp y$ and $x\limp z$).
Although not strictly necessary, clauses
for other facts about the white entries in the first row of $P$,
such as $P_{0,1,2}\limp P_{0,2,3}$, are also included.

Enforcing the fact that rows are sorted by transversal type is done with the cardinality constraints
discussed in Section~\ref{sec:colour}, as these constraints allow us to fix which rows are of which types.
For example, suppose that $P$ is of type R, meaning that $P$ consists of eight transversals of type $p_1$
and two transversals of type $p_4$.  Then we would enforce the first eight rows of $P$ to be of type $p_1$
with $P_{i,6,\w}+\dotsb+P_{i,9,\w}=1$ and $P_{i,0,\d}+\dotsb+P_{i,5,\d}=0$ for $0\leq i<8$, and the last two rows of $P$ to be of type $p_4$
with $P_{i,6,\w}+\dotsb+P_{i,9,\w}=4$ and $P_{i,0,\d}+\dotsb+P_{i,5,\d}=6$ for $i=8$ and~$9$.

Finally, we enforce that rows with the same transversal type in $P$ are sorted in lexicographic order
by ensuring their initial entries are increasing.  For example, suppose rows $i$ and $i+1$
of $P$ have the same transversal type.  Then we add the constraint
$P_{i,0,k}\limp\lnot P_{i+1,0,l}$ for all $0\leq l<k<n$, which says that the initial
entry of row $i+1$ of $P$ cannot be smaller than the initial entry of row $i$.
We add the same constraints for $Q$ as well.

\subsection{Postprocessing}\label{sec:extendability}

As we will describe in Section~\ref{sec:results}, the encoding presented thus far
successfully found many TRPs $(P,Q)$ corresponding to Myrvold's eight unsolved cases.
We performed some postprocessing on these pairs to check if they were extendable to
a triple of mutual transversal representations and also to check the pairs for
equivalence.

First, we used a SAT solver to check all pairs $(P,Q)$ for extendability to a triple.
This was done by creating new SAT instances for each pair encoding
both squares $P$ and $Q$, along with a new Latin square $L$, and then asserting that $(L,P)$
is a TRP and $(L,Q)$ is a TRP by using the encoding described in Section~\ref{sec:TRPC}
twice.  The entries of $P$ and $Q$ were specified using unit clauses; i.e., if
$P[i,j]=k$ then the clause $P_{i,j,k}$ was added to the SAT instance.  Because
of the presence of so many unit clauses these instances were highly constrained
and in all cases were shown by the SAT solver to be unsatisfiable within 0.1 seconds.
Thus, no pairs we found were extendable to a triple.  However, this does not eliminate
the possibility that there might exist a triple $(P,Q,L)$ corresponding to some of
Myrvold's cases, because we did not exhaustively enumerate all $(P,Q)$s for any
of Myrvold's unsolved types.

Finally, we checked all the TRPs $(P,Q)$ that we found to see if any were equivalent
to each other.  This was done by converting the TRP into its orthogonal pair representation
$(P^{-1}Q,Q)$, reducing the orthogonal pair to a graph using the reduction given
by Egan and Wanless~\cite{Egan2015}, and finally checking the graphs for equivalence
using the graph isomorphism tool \textsc{nauty}~\cite{mckay2014practical}.

Precisely, the reduction from a $\MOLS{(k-2)}{n}$ to a graph is described using
what is known as an orthogonal array.  An orthogonal array for a $\MOLS{(k-2)}{n}$ is a matrix $O$ of
size $n^2 \times k$, with entries in $\{0,\dotsc,n-1\}$,
with every possible pair of symbols appearing exactly once in any two columns of $O$.
Define an undirected graph $G_O$ corresponding to $O$. The vertices of $G_O$ are of three types:
\begin{itemize}
    \item[$\bullet$] $k$ type $1$ vertices that correspond to the columns of $O$,
    \item[$\bullet$] $kn$ type $2$ vertices that correspond to the symbols in each of the columns of $O$, and
    \item[$\bullet$] $n^2$ type $3$ vertices that correspond to the rows of $O$.
\end{itemize}
Each type $1$ vertex is joined to the $n$ type $2$ vertices that correspond to the symbols in its column. Each type $3$ vertex is connected to the $k$ type $2$ vertices that correspond to the symbols in its row.
Vertices are coloured according to their type so that isomorphisms are not allowed to change the type of a vertex.

After forming the graphs corresponding to all TRPs $(P,Q)$ we found, \textsc{nauty} determined that no two graphs were
isomorphic.  Thus, we have confirmation that the SAT solver is indeed exploring different parts of the search space
and that multiple inequivalent TRPs exist corresponding to Myrvold's unsolved cases.
However, we did not attempt to perform an exhaustive
search for TRPs in any of Myrvold's unsolved cases.
Given the enormity of the search space, and the fact that no solutions were repeated even after several hundred
solutions had already been found, we suspect that an exhaustive search would require
a huge amount of additional computational resources or at least some more restrictive properties
that could be applied to Myrvold's unsolved cases.

\section{Results}\label{sec:results}

We now discuss the results of our computational investigation into
Myrvold's results.  The computations were performed using the SAT solver
Kissat 4.0.4~\cite{biere2022gimsatul} run on AMD EPYC Zen~5
processors running at 2.7~GHz and equipped with
1~GiB of memory.

Recall Myrvold showed~\cite[Thm~4.4]{myrvold1999negative},
if $P$ and $Q$ are both transversal representations of a Latin square of order ten containing
a subsquare of order four, then up to ordering there are twenty-eight possible cases
for $P$ and $Q$ and twenty of these cases can be ruled out.
The eight possible cases Myrvold left remaining are
$(\S,\X)$, $(\U,\U)$, $(\U,\W)$, $(\U,\X)$, $(\V,\X)$, $(\W,\W)$, $(\W,\X)$, and $(\X,\X)$.

We used our SAT encoding to generate twenty-eight SAT instances, one for each of Myrvold's cases.
The twenty cases ruled out by Myrvold were each found
to be unsatisfiable in under 0.2 seconds.  The eight cases left
open by Myrvold were all considerably harder to solve, but 
each was found to be satisfiable,
explaining why Myrvold was unable
to eliminate these eight cases from consideration.
Kissat stops solving as soon
as it finds a satisfying assignment of the provided instance,
and we use the satisfying assignment reported by Kissat to form
a coloured TRP in each of the eight cases (see the \hyperref[sec:appendix]{Appendix}
for explicit examples of TRPs in each case).

Because the satisfiable cases were significantly more difficult
than the unsatisfiable cases, we found it useful to exploit parallelization
when solving the satisfiable instances.  We started 49 independent Kissat processes
for each satisfiable case
and each process was run on one processor core for up to one week.  Each process
was provided with a different random seed, so no two copies of Kissat would
make the same choices during the solving process.
Each process was terminated if Kissat did not find a solution within a week.
Results from these searches are available in Table~\ref{tbl:summary},
and a scatterplot of the running times is given in Figure~\ref{fig:runtimes}.
There is a significant amount of variance in the running times,
but in general the case $(\U,\U)$ was the easiest to solve and the case $(\X,\X)$
was the hardest to solve.

We summarize some statistical information about the TRPs we found in Table~\ref{tbl:stats}.
In particular, for each pair type we provide the number of TRPs found that are
compatible with the $4\times4$ subsquares $\Omega_1$ and $\Omega_2$ in $L$.
In case $(\V,\X)$, the solver was able to show there are no TRPs consistent
with the choice $\Omega_1$ in under 0.2 seconds.  This can be explained by the fact that
the square~$\Omega_1$ has no transversals---it follows that~$\Omega_1$ is inconsistent
with square type~$\V$, because the white entries in a row of type $p_4$
must represent a transversal in $\Omega$.

Usually the TRPs we found were consistent with only one of $\Omega_1$ or~$\Omega_2$,
but two TRPs were consistent
with both choices of~$\Omega$ simultaneously.  Both were of type $(\X,\X)$ and
one of these TRPs is provided as the example $(\X,\X)$ pair in the appendix.
Also listed in Table~\ref{tbl:stats} are the minimum and maximum
number of transversals and mates in each of the
squares in the TRPs we found.
It also reports on the number of
\emph{common} transversals in the TRPs (i.e., transversals of both squares in the TRP whose
row representation is the same in both).
Most TRPs had no common transversals,
and none had more than two common transversals.
This is an indication that the TRPs we found
are not very close to extending to a triple of mutual TRPs,
since for $(P,Q)$ to extend to a triple of mutual
TRPs, $P$ and $Q$ must have at least $n$ common transversals.

\begin{table}
\caption{A summary of the running times (in seconds) of the instances
for each of the eight pair types with solutions.
Each pair type had 49 independently-solved SAT instances and
were run with a one week timeout.
The timeouts were included in the computation of each statistic and
counted as running for a full week.}\label{tbl:summary}
\begin{center}
\begin{tabular}{c@{\quad}c@{\quad}c@{\quad}c@{\quad}c}
pair type           &       mean &     median &        min &        max \\
$(\U,\U)$           &  \031102.1 &  \019619.4 &    \0748.6 &  \098009.2 \\
$(\S,\X)$           &  \058780.5 &  \038453.2 &     2282.4 &   175005.1 \\
$(\U,\W)$           &  \075043.9 &  \056171.4 &     2659.8 &   399428.7 \\
$(\W,\W)$           &   139198.5 &  \097661.6 &     1662.1 &    timeout \\
$(\V,\X)$           &   147169.2 &   114191.0 &     2567.7 &    timeout \\
$(\U,\X)$           &   140560.4 &   117378.6 &    \0327.8 &    timeout \\
$(\W,\X)$           &   222515.7 &   176970.4 &    \0527.3 &    timeout \\
$(\X,\X)$           &   429809.6 &   580524.5 &     6747.1 &    timeout \\
\end{tabular}
\end{center}
\end{table}

\begin{table}
\caption{A summary of the TRPs we found using
49 independently-solved SAT instances for each pair type.  The table
includes the number of solved SAT instances, the number of TRPs
compatible with the $4\times4$ subsquares $\Omega_1$ and $\Omega_2$,
and the minimum and maximum number of transversals, mates,
and common transversals appearing in the TRPs.}\label{tbl:stats}
\begin{center}
\begin{tabular}{c@{\quad}c@{\quad}c@{\quad}c@{\quad}c@{\quad}c@{\quad}c}
pair type           & \# solved      & \#$\Omega_1$ & \#$\Omega_2$ & transversals & mates & common trans. \\
$(\U,\U)$ &            49&             9&            40&      776--900&          1--6&          0--1 \\
$(\S,\X)$ &            49&            27&            22&      768--948&          1--7&          0--1 \\
$(\U,\W)$ &            49&            19&            30&      744--912&          1--5&          0--0 \\
$(\W,\W)$ &            48&            25&            23&      764--900&          1--5&          0--1 \\
$(\V,\X)$ &            48&             0&            48&      756--940&          1--8&          0--2 \\
$(\U,\X)$ &            48&            20&            28&      724--924&          1--6&          0--1 \\
$(\W,\X)$ &            46&            23&            23&      772--924&          1--9&          0--2 \\
$(\X,\X)$ &            25&            13&            14&      772--912&          1--5&          0--1 \\
\end{tabular}
\end{center}
\end{table}

\begin{figure}
\centering
\includegraphics[width=0.9\textwidth]{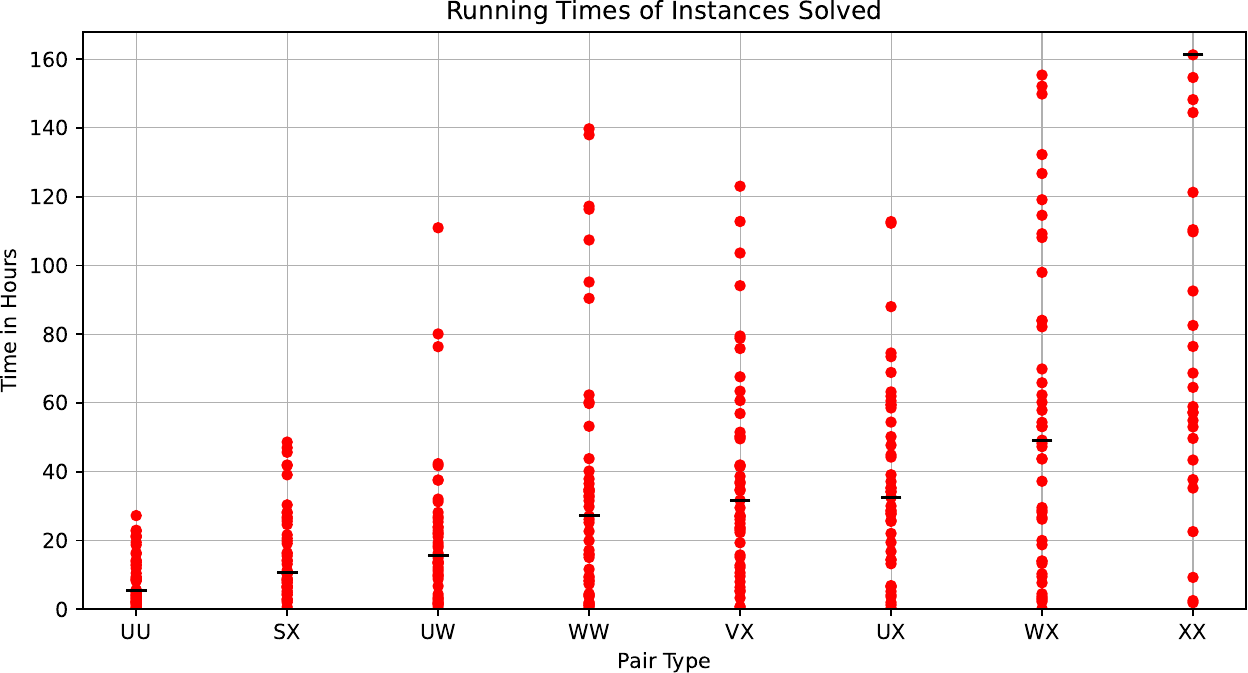}
\caption{A scatterplot of the solver's running time for each pair type.
The median running time is shown as a solid black line.
Timeouts are not plotted
but are used in determining the median.}
\label{fig:runtimes}
\end{figure}

\section{Conclusion}\label{Conclusion}

In this paper we use a satisfiability (SAT) solver to investigate Myrvold's nonexistence results~\cite{myrvold1999negative}
on orthogonal triples of Latin squares of order ten.  The SAT solver almost instantaneously
rules out the cases that Myrvold ruled out, and more significantly, the SAT solver
provides explicit examples of Latin square pairs in each of the cases that
Myrvold was unable to rule out---providing an explanation for why Myrvold
was unable to rule out these cases and determining a negative resolution
to the following question left open by Myrvold:
\begin{quote}
\emph{Possibly, with a bit more ingenuity, the remaining cases can be eliminated.}
\end{quote}
We show that pairs exist in the remaining cases, and so eliminating the
remaining cases with ``a bit more ingenuity'' is probably
not achievable---at the very least,
any argument required to eliminate the remaining cases
would need to be more sophisticated in having to
rely on the existence of the third square,~$L$.
We were also able to show that requiring compatibility
with the $4\times4$ Latin subsquare in~$L$ is not by itself
sufficient to rule out any of the remaining cases.
It would be interesting to know if some of the remaining cases
could be ruled out by considering additional structure in
$L$, but we leave this as future work.

In order to derive a concise and effective SAT encoding for our
search we make use of a duality between orthogonal Latin squares
and transversal representation pairs.  Although such a duality has
long been used in searches for Latin squares, we also give an explicit formulation
of how this duality arises via a composition operation on Latin squares.
We found this viewpoint useful when deriving our encoding
and surprisingly we were not able to find it expressed in prior literature.

\paragraph{Acknowledgements}
We thank the reviewers for their detailed feedback which improved the
paper.  In particular, a reviewer pointed out the possibility of adding
constraints enforcing that the transversal representation pair is
consistent with the $4\times4$ Latin subsquare in $L$.  We also
thank Tanbir Ahmed for his help during the editing process.

\bibliographystyle{splncs04}
\bibliography{bibliography}
 
\section*{Appendix}\label{sec:appendix}

In the appendix we provide eight explicit pairs we found which prove the existence of TRPs
for Myrvold's eight unresolved cases~\cite{myrvold1999negative}.

\begin{figure}[ht]
  \begin{minipage}[c][1\width]{0.44\textwidth}
  \centering
  \begin{tikzpicture}
  \matrix (m) [matrix of nodes,nodes={element},column sep=-\pgflinewidth, row sep=-\pgflinewidth, label={[font=\large]above: type S}]{
  |[draw,fill=white]| 0 &
  |[draw,fill=white]| 1 &
  |[draw,fill=white]| 3 &
  |[draw,fill=light]| \textcolor{black}{4} &
  |[draw,fill=light]| \textcolor{black}{5} &
  |[draw,fill=light]| \textcolor{black}{6} &
  |[draw,fill=white]| 2 &
  |[draw,fill=light]| \textcolor{black}{7} &
  |[draw,fill=light]| \textcolor{black}{8} &
  |[draw,fill=light]| \textcolor{black}{9} \\
  |[draw,fill=white]| 1 &
  |[draw,fill=light]| \textcolor{black}{6} &
  |[draw,fill=light]| \textcolor{black}{9} &
  |[draw,fill=white]| 2 &
  |[draw,fill=white]| 3 &
  |[draw,fill=light]| \textcolor{black}{4} &
  |[draw,fill=white]| 0 &
  |[draw,fill=light]| \textcolor{black}{8} &
  |[draw,fill=light]| \textcolor{black}{5} &
  |[draw,fill=light]| \textcolor{black}{7} \\
  |[draw,fill=white]| 3 &
  |[draw,fill=light]| \textcolor{black}{8} &
  |[draw,fill=white]| 2 &
  |[draw,fill=white]| 1 &
  |[draw,fill=light]| \textcolor{black}{7} &
  |[draw,fill=light]| \textcolor{black}{5} &
  |[draw,fill=light]| \textcolor{black}{6} &
  |[draw,fill=light]| \textcolor{black}{4} &
  |[draw,fill=light]| \textcolor{black}{9} &
  |[draw,fill=white]| 0 \\
  |[draw,fill=light]| \textcolor{black}{4} &
  |[draw,fill=white]| 2 &
  |[draw,fill=light]| \textcolor{black}{7} &
  |[draw,fill=white]| 0 &
  |[draw,fill=white]| 1 &
  |[draw,fill=light]| \textcolor{black}{8} &
  |[draw,fill=white]| 3 &
  |[draw,fill=light]| \textcolor{black}{9} &
  |[draw,fill=light]| \textcolor{black}{6} &
  |[draw,fill=light]| \textcolor{black}{5} \\
  |[draw,fill=light]| \textcolor{black}{5} &
  |[draw,fill=white]| 0 &
  |[draw,fill=light]| \textcolor{black}{8} &
  |[draw,fill=light]| \textcolor{black}{9} &
  |[draw,fill=white]| 2 &
  |[draw,fill=white]| 3 &
  |[draw,fill=light]| \textcolor{black}{7} &
  |[draw,fill=white]| 1 &
  |[draw,fill=light]| \textcolor{black}{4} &
  |[draw,fill=light]| \textcolor{black}{6} \\
  |[draw,fill=light]| \textcolor{black}{7} &
  |[draw,fill=white]| 3 &
  |[draw,fill=white]| 0 &
  |[draw,fill=light]| \textcolor{black}{5} &
  |[draw,fill=light]| \textcolor{black}{9} &
  |[draw,fill=white]| 1 &
  |[draw,fill=light]| \textcolor{black}{4} &
  |[draw,fill=light]| \textcolor{black}{6} &
  |[draw,fill=white]| 2 &
  |[draw,fill=light]| \textcolor{black}{8} \\
  |[draw,fill=light]| \textcolor{black}{8} &
  |[draw,fill=light]| \textcolor{black}{5} &
  |[draw,fill=white]| 1 &
  |[draw,fill=light]| \textcolor{black}{6} &
  |[draw,fill=white]| 0 &
  |[draw,fill=white]| 2 &
  |[draw,fill=light]| \textcolor{black}{9} &
  |[draw,fill=white]| 3 &
  |[draw,fill=light]| \textcolor{black}{7} &
  |[draw,fill=light]| \textcolor{black}{4} \\
  |[draw,fill=white]| 2 &
  |[draw,fill=dark]| \textcolor{white}{4} &
  |[draw,fill=light]| \textcolor{black}{5} &
  |[draw,fill=dark]| \textcolor{white}{7} &
  |[draw,fill=dark]| \textcolor{white}{6} &
  |[draw,fill=dark]| \textcolor{white}{9} &
  |[draw,fill=light]| \textcolor{black}{8} &
  |[draw,fill=white]| 0 &
  |[draw,fill=white]| 3 &
  |[draw,fill=white]| 1 \\
  |[draw,fill=dark]| \textcolor{white}{6} &
  |[draw,fill=light]| \textcolor{black}{9} &
  |[draw,fill=dark]| \textcolor{white}{4} &
  |[draw,fill=white]| 3 &
  |[draw,fill=dark]| \textcolor{white}{8} &
  |[draw,fill=dark]| \textcolor{white}{7} &
  |[draw,fill=white]| 1 &
  |[draw,fill=light]| \textcolor{black}{5} &
  |[draw,fill=white]| 0 &
  |[draw,fill=white]| 2 \\
  |[draw,fill=dark]| \textcolor{white}{9} &
  |[draw,fill=dark]| \textcolor{white}{7} &
  |[draw,fill=dark]| \textcolor{white}{6} &
  |[draw,fill=dark]| \textcolor{white}{8} &
  |[draw,fill=light]| \textcolor{black}{4} &
  |[draw,fill=white]| 0 &
  |[draw,fill=light]| \textcolor{black}{5} &
  |[draw,fill=white]| 2 &
  |[draw,fill=white]| 1 &
  |[draw,fill=white]| 3 \\
  };\end{tikzpicture}
  \end{minipage}
  \hfill
  \begin{minipage}[c][1\width]{0.44\textwidth}
  \centering
  \begin{tikzpicture}
  \matrix (m) [matrix of nodes,nodes={element},column sep=-\pgflinewidth, row sep=-\pgflinewidth, label={[font=\large]above: type X}]{
  |[draw,fill=white]| 2 &
  |[draw,fill=light]| \textcolor{black}{6} &
  |[draw,fill=white]| 1 &
  |[draw,fill=light]| \textcolor{black}{4} &
  |[draw,fill=light]| \textcolor{black}{7} &
  |[draw,fill=white]| 3 &
  |[draw,fill=light]| \textcolor{black}{5} &
  |[draw,fill=light]| \textcolor{black}{9} &
  |[draw,fill=white]| 0 &
  |[draw,fill=light]| \textcolor{black}{8} \\
  |[draw,fill=light]| \textcolor{black}{4} &
  |[draw,fill=light]| \textcolor{black}{9} &
  |[draw,fill=white]| 3 &
  |[draw,fill=light]| \textcolor{black}{5} &
  |[draw,fill=white]| 2 &
  |[draw,fill=white]| 0 &
  |[draw,fill=light]| \textcolor{black}{6} &
  |[draw,fill=light]| \textcolor{black}{8} &
  |[draw,fill=light]| \textcolor{black}{7} &
  |[draw,fill=white]| 1 \\
  |[draw,fill=light]| \textcolor{black}{7} &
  |[draw,fill=white]| 1 &
  |[draw,fill=white]| 2 &
  |[draw,fill=light]| \textcolor{black}{9} &
  |[draw,fill=white]| 0 &
  |[draw,fill=light]| \textcolor{black}{4} &
  |[draw,fill=light]| \textcolor{black}{8} &
  |[draw,fill=light]| \textcolor{black}{5} &
  |[draw,fill=light]| \textcolor{black}{6} &
  |[draw,fill=white]| 3 \\
  |[draw,fill=light]| \textcolor{black}{8} &
  |[draw,fill=white]| 0 &
  |[draw,fill=light]| \textcolor{black}{5} &
  |[draw,fill=white]| 3 &
  |[draw,fill=white]| 1 &
  |[draw,fill=light]| \textcolor{black}{6} &
  |[draw,fill=light]| \textcolor{black}{4} &
  |[draw,fill=white]| 2 &
  |[draw,fill=light]| \textcolor{black}{9} &
  |[draw,fill=light]| \textcolor{black}{7} \\
  |[draw,fill=white]| 0 &
  |[draw,fill=light]| \textcolor{black}{5} &
  |[draw,fill=light]| \textcolor{black}{7} &
  |[draw,fill=dark]| \textcolor{white}{8} &
  |[draw,fill=white]| 3 &
  |[draw,fill=dark]| \textcolor{white}{9} &
  |[draw,fill=white]| 1 &
  |[draw,fill=light]| \textcolor{black}{4} &
  |[draw,fill=white]| 2 &
  |[draw,fill=light]| \textcolor{black}{6} \\
  |[draw,fill=white]| 1 &
  |[draw,fill=dark]| \textcolor{white}{4} &
  |[draw,fill=dark]| \textcolor{white}{6} &
  |[draw,fill=white]| 0 &
  |[draw,fill=light]| \textcolor{black}{9} &
  |[draw,fill=light]| \textcolor{black}{5} &
  |[draw,fill=light]| \textcolor{black}{7} &
  |[draw,fill=white]| 3 &
  |[draw,fill=light]| \textcolor{black}{8} &
  |[draw,fill=white]| 2 \\
  |[draw,fill=white]| 3 &
  |[draw,fill=dark]| \textcolor{white}{7} &
  |[draw,fill=light]| \textcolor{black}{9} &
  |[draw,fill=light]| \textcolor{black}{6} &
  |[draw,fill=dark]| \textcolor{white}{8} &
  |[draw,fill=white]| 1 &
  |[draw,fill=white]| 2 &
  |[draw,fill=white]| 0 &
  |[draw,fill=light]| \textcolor{black}{4} &
  |[draw,fill=light]| \textcolor{black}{5} \\
  |[draw,fill=light]| \textcolor{black}{5} &
  |[draw,fill=white]| 3 &
  |[draw,fill=dark]| \textcolor{white}{4} &
  |[draw,fill=white]| 2 &
  |[draw,fill=dark]| \textcolor{white}{6} &
  |[draw,fill=light]| \textcolor{black}{8} &
  |[draw,fill=light]| \textcolor{black}{9} &
  |[draw,fill=light]| \textcolor{black}{7} &
  |[draw,fill=white]| 1 &
  |[draw,fill=white]| 0 \\
  |[draw,fill=dark]| \textcolor{white}{6} &
  |[draw,fill=light]| \textcolor{black}{8} &
  |[draw,fill=white]| 0 &
  |[draw,fill=dark]| \textcolor{white}{7} &
  |[draw,fill=light]| \textcolor{black}{4} &
  |[draw,fill=white]| 2 &
  |[draw,fill=white]| 3 &
  |[draw,fill=white]| 1 &
  |[draw,fill=light]| \textcolor{black}{5} &
  |[draw,fill=light]| \textcolor{black}{9} \\
  |[draw,fill=dark]| \textcolor{white}{9} &
  |[draw,fill=white]| 2 &
  |[draw,fill=light]| \textcolor{black}{8} &
  |[draw,fill=white]| 1 &
  |[draw,fill=light]| \textcolor{black}{5} &
  |[draw,fill=dark]| \textcolor{white}{7} &
  |[draw,fill=white]| 0 &
  |[draw,fill=light]| \textcolor{black}{6} &
  |[draw,fill=white]| 3 &
  |[draw,fill=light]| \textcolor{black}{4} \\
  };\end{tikzpicture}
  \end{minipage}
\end{figure}
\begin{figure}[ht]
  \begin{minipage}[c][1\width]{0.44\textwidth}
  \centering
  \begin{tikzpicture}
  \matrix (m) [matrix of nodes,nodes={element},column sep=-\pgflinewidth, row sep=-\pgflinewidth, label={[font=\large]above: type U}]{
  |[draw,fill=white]| 0 &
  |[draw,fill=white]| 1 &
  |[draw,fill=white]| 2 &
  |[draw,fill=light]| \textcolor{black}{4} &
  |[draw,fill=light]| \textcolor{black}{5} &
  |[draw,fill=light]| \textcolor{black}{6} &
  |[draw,fill=white]| 3 &
  |[draw,fill=light]| \textcolor{black}{7} &
  |[draw,fill=light]| \textcolor{black}{8} &
  |[draw,fill=light]| \textcolor{black}{9} \\
  |[draw,fill=white]| 1 &
  |[draw,fill=light]| \textcolor{black}{7} &
  |[draw,fill=light]| \textcolor{black}{8} &
  |[draw,fill=white]| 0 &
  |[draw,fill=light]| \textcolor{black}{4} &
  |[draw,fill=white]| 2 &
  |[draw,fill=light]| \textcolor{black}{5} &
  |[draw,fill=light]| \textcolor{black}{9} &
  |[draw,fill=white]| 3 &
  |[draw,fill=light]| \textcolor{black}{6} \\
  |[draw,fill=white]| 2 &
  |[draw,fill=light]| \textcolor{black}{9} &
  |[draw,fill=light]| \textcolor{black}{7} &
  |[draw,fill=white]| 1 &
  |[draw,fill=light]| \textcolor{black}{8} &
  |[draw,fill=white]| 3 &
  |[draw,fill=white]| 0 &
  |[draw,fill=light]| \textcolor{black}{4} &
  |[draw,fill=light]| \textcolor{black}{6} &
  |[draw,fill=light]| \textcolor{black}{5} \\
  |[draw,fill=white]| 3 &
  |[draw,fill=light]| \textcolor{black}{5} &
  |[draw,fill=white]| 0 &
  |[draw,fill=light]| \textcolor{black}{7} &
  |[draw,fill=white]| 2 &
  |[draw,fill=light]| \textcolor{black}{4} &
  |[draw,fill=light]| \textcolor{black}{6} &
  |[draw,fill=light]| \textcolor{black}{8} &
  |[draw,fill=light]| \textcolor{black}{9} &
  |[draw,fill=white]| 1 \\
  |[draw,fill=light]| \textcolor{black}{4} &
  |[draw,fill=white]| 2 &
  |[draw,fill=light]| \textcolor{black}{6} &
  |[draw,fill=white]| 3 &
  |[draw,fill=white]| 0 &
  |[draw,fill=light]| \textcolor{black}{9} &
  |[draw,fill=light]| \textcolor{black}{7} &
  |[draw,fill=white]| 1 &
  |[draw,fill=light]| \textcolor{black}{5} &
  |[draw,fill=light]| \textcolor{black}{8} \\
  |[draw,fill=light]| \textcolor{black}{9} &
  |[draw,fill=light]| \textcolor{black}{6} &
  |[draw,fill=white]| 1 &
  |[draw,fill=white]| 2 &
  |[draw,fill=white]| 3 &
  |[draw,fill=light]| \textcolor{black}{7} &
  |[draw,fill=light]| \textcolor{black}{8} &
  |[draw,fill=light]| \textcolor{black}{5} &
  |[draw,fill=white]| 0 &
  |[draw,fill=light]| \textcolor{black}{4} \\
  |[draw,fill=dark]| \textcolor{white}{5} &
  |[draw,fill=dark]| \textcolor{white}{8} &
  |[draw,fill=white]| 3 &
  |[draw,fill=light]| \textcolor{black}{6} &
  |[draw,fill=light]| \textcolor{black}{9} &
  |[draw,fill=white]| 1 &
  |[draw,fill=light]| \textcolor{black}{4} &
  |[draw,fill=white]| 0 &
  |[draw,fill=white]| 2 &
  |[draw,fill=light]| \textcolor{black}{7} \\
  |[draw,fill=light]| \textcolor{black}{8} &
  |[draw,fill=white]| 3 &
  |[draw,fill=light]| \textcolor{black}{4} &
  |[draw,fill=dark]| \textcolor{white}{9} &
  |[draw,fill=white]| 1 &
  |[draw,fill=dark]| \textcolor{white}{5} &
  |[draw,fill=white]| 2 &
  |[draw,fill=light]| \textcolor{black}{6} &
  |[draw,fill=light]| \textcolor{black}{7} &
  |[draw,fill=white]| 0 \\
  |[draw,fill=light]| \textcolor{black}{6} &
  |[draw,fill=white]| 0 &
  |[draw,fill=dark]| \textcolor{white}{9} &
  |[draw,fill=dark]| \textcolor{white}{5} &
  |[draw,fill=dark]| \textcolor{white}{7} &
  |[draw,fill=dark]| \textcolor{white}{8} &
  |[draw,fill=white]| 1 &
  |[draw,fill=white]| 2 &
  |[draw,fill=light]| \textcolor{black}{4} &
  |[draw,fill=white]| 3 \\
  |[draw,fill=dark]| \textcolor{white}{7} &
  |[draw,fill=dark]| \textcolor{white}{4} &
  |[draw,fill=dark]| \textcolor{white}{5} &
  |[draw,fill=light]| \textcolor{black}{8} &
  |[draw,fill=dark]| \textcolor{white}{6} &
  |[draw,fill=white]| 0 &
  |[draw,fill=light]| \textcolor{black}{9} &
  |[draw,fill=white]| 3 &
  |[draw,fill=white]| 1 &
  |[draw,fill=white]| 2 \\
  };\end{tikzpicture}
  \end{minipage}
  \hfill
  \begin{minipage}[c][1\width]{0.44\textwidth}
  \centering
  \begin{tikzpicture}
  \matrix (m) [matrix of nodes,nodes={element},column sep=-\pgflinewidth, row sep=-\pgflinewidth, label={[font=\large]above: type U}]{
  |[draw,fill=white]| 0 &
  |[draw,fill=white]| 3 &
  |[draw,fill=light]| \textcolor{black}{6} &
  |[draw,fill=white]| 1 &
  |[draw,fill=light]| \textcolor{black}{9} &
  |[draw,fill=light]| \textcolor{black}{7} &
  |[draw,fill=light]| \textcolor{black}{5} &
  |[draw,fill=light]| \textcolor{black}{8} &
  |[draw,fill=light]| \textcolor{black}{4} &
  |[draw,fill=white]| 2 \\
  |[draw,fill=white]| 2 &
  |[draw,fill=light]| \textcolor{black}{5} &
  |[draw,fill=white]| 1 &
  |[draw,fill=light]| \textcolor{black}{8} &
  |[draw,fill=white]| 0 &
  |[draw,fill=light]| \textcolor{black}{6} &
  |[draw,fill=light]| \textcolor{black}{4} &
  |[draw,fill=light]| \textcolor{black}{9} &
  |[draw,fill=light]| \textcolor{black}{7} &
  |[draw,fill=white]| 3 \\
  |[draw,fill=white]| 3 &
  |[draw,fill=white]| 0 &
  |[draw,fill=light]| \textcolor{black}{4} &
  |[draw,fill=light]| \textcolor{black}{6} &
  |[draw,fill=light]| \textcolor{black}{8} &
  |[draw,fill=white]| 2 &
  |[draw,fill=light]| \textcolor{black}{7} &
  |[draw,fill=light]| \textcolor{black}{5} &
  |[draw,fill=white]| 1 &
  |[draw,fill=light]| \textcolor{black}{9} \\
  |[draw,fill=light]| \textcolor{black}{6} &
  |[draw,fill=light]| \textcolor{black}{9} &
  |[draw,fill=white]| 2 &
  |[draw,fill=white]| 0 &
  |[draw,fill=white]| 1 &
  |[draw,fill=light]| \textcolor{black}{4} &
  |[draw,fill=light]| \textcolor{black}{8} &
  |[draw,fill=white]| 3 &
  |[draw,fill=light]| \textcolor{black}{5} &
  |[draw,fill=light]| \textcolor{black}{7} \\
  |[draw,fill=light]| \textcolor{black}{8} &
  |[draw,fill=light]| \textcolor{black}{7} &
  |[draw,fill=white]| 0 &
  |[draw,fill=white]| 3 &
  |[draw,fill=light]| \textcolor{black}{5} &
  |[draw,fill=white]| 1 &
  |[draw,fill=light]| \textcolor{black}{9} &
  |[draw,fill=white]| 2 &
  |[draw,fill=light]| \textcolor{black}{6} &
  |[draw,fill=light]| \textcolor{black}{4} \\
  |[draw,fill=light]| \textcolor{black}{9} &
  |[draw,fill=white]| 2 &
  |[draw,fill=white]| 3 &
  |[draw,fill=light]| \textcolor{black}{7} &
  |[draw,fill=light]| \textcolor{black}{4} &
  |[draw,fill=white]| 0 &
  |[draw,fill=white]| 1 &
  |[draw,fill=light]| \textcolor{black}{6} &
  |[draw,fill=light]| \textcolor{black}{8} &
  |[draw,fill=light]| \textcolor{black}{5} \\
  |[draw,fill=white]| 1 &
  |[draw,fill=light]| \textcolor{black}{6} &
  |[draw,fill=dark]| \textcolor{white}{5} &
  |[draw,fill=light]| \textcolor{black}{4} &
  |[draw,fill=dark]| \textcolor{white}{7} &
  |[draw,fill=white]| 3 &
  |[draw,fill=white]| 2 &
  |[draw,fill=white]| 0 &
  |[draw,fill=light]| \textcolor{black}{9} &
  |[draw,fill=light]| \textcolor{black}{8} \\
  |[draw,fill=dark]| \textcolor{white}{7} &
  |[draw,fill=white]| 1 &
  |[draw,fill=light]| \textcolor{black}{8} &
  |[draw,fill=dark]| \textcolor{white}{5} &
  |[draw,fill=white]| 3 &
  |[draw,fill=light]| \textcolor{black}{9} &
  |[draw,fill=light]| \textcolor{black}{6} &
  |[draw,fill=light]| \textcolor{black}{4} &
  |[draw,fill=white]| 2 &
  |[draw,fill=white]| 0 \\
  |[draw,fill=light]| \textcolor{black}{4} &
  |[draw,fill=dark]| \textcolor{white}{8} &
  |[draw,fill=dark]| \textcolor{white}{9} &
  |[draw,fill=white]| 2 &
  |[draw,fill=dark]| \textcolor{white}{6} &
  |[draw,fill=dark]| \textcolor{white}{5} &
  |[draw,fill=white]| 0 &
  |[draw,fill=light]| \textcolor{black}{7} &
  |[draw,fill=white]| 3 &
  |[draw,fill=white]| 1 \\
  |[draw,fill=dark]| \textcolor{white}{5} &
  |[draw,fill=dark]| \textcolor{white}{4} &
  |[draw,fill=light]| \textcolor{black}{7} &
  |[draw,fill=dark]| \textcolor{white}{9} &
  |[draw,fill=white]| 2 &
  |[draw,fill=dark]| \textcolor{white}{8} &
  |[draw,fill=white]| 3 &
  |[draw,fill=white]| 1 &
  |[draw,fill=white]| 0 &
  |[draw,fill=light]| \textcolor{black}{6} \\
  };\end{tikzpicture}
  \end{minipage}
\end{figure}
\begin{figure}[ht]
  \begin{minipage}[c][1\width]{0.44\textwidth}
  \centering
  \begin{tikzpicture}
  \matrix (m) [matrix of nodes,nodes={element},column sep=-\pgflinewidth, row sep=-\pgflinewidth, label={[font=\large]above: type U}]{
  |[draw,fill=white]| 0 &
  |[draw,fill=white]| 2 &
  |[draw,fill=white]| 3 &
  |[draw,fill=light]| \textcolor{black}{4} &
  |[draw,fill=light]| \textcolor{black}{5} &
  |[draw,fill=light]| \textcolor{black}{6} &
  |[draw,fill=white]| 1 &
  |[draw,fill=light]| \textcolor{black}{7} &
  |[draw,fill=light]| \textcolor{black}{8} &
  |[draw,fill=light]| \textcolor{black}{9} \\
  |[draw,fill=white]| 1 &
  |[draw,fill=light]| \textcolor{black}{4} &
  |[draw,fill=light]| \textcolor{black}{6} &
  |[draw,fill=white]| 0 &
  |[draw,fill=light]| \textcolor{black}{9} &
  |[draw,fill=white]| 2 &
  |[draw,fill=white]| 3 &
  |[draw,fill=light]| \textcolor{black}{8} &
  |[draw,fill=light]| \textcolor{black}{5} &
  |[draw,fill=light]| \textcolor{black}{7} \\
  |[draw,fill=white]| 2 &
  |[draw,fill=white]| 3 &
  |[draw,fill=light]| \textcolor{black}{5} &
  |[draw,fill=light]| \textcolor{black}{9} &
  |[draw,fill=white]| 0 &
  |[draw,fill=light]| \textcolor{black}{8} &
  |[draw,fill=light]| \textcolor{black}{7} &
  |[draw,fill=light]| \textcolor{black}{4} &
  |[draw,fill=light]| \textcolor{black}{6} &
  |[draw,fill=white]| 1 \\
  |[draw,fill=light]| \textcolor{black}{5} &
  |[draw,fill=light]| \textcolor{black}{6} &
  |[draw,fill=light]| \textcolor{black}{7} &
  |[draw,fill=white]| 1 &
  |[draw,fill=white]| 3 &
  |[draw,fill=white]| 0 &
  |[draw,fill=light]| \textcolor{black}{8} &
  |[draw,fill=white]| 2 &
  |[draw,fill=light]| \textcolor{black}{9} &
  |[draw,fill=light]| \textcolor{black}{4} \\
  |[draw,fill=light]| \textcolor{black}{6} &
  |[draw,fill=white]| 0 &
  |[draw,fill=white]| 1 &
  |[draw,fill=white]| 2 &
  |[draw,fill=light]| \textcolor{black}{8} &
  |[draw,fill=light]| \textcolor{black}{5} &
  |[draw,fill=light]| \textcolor{black}{4} &
  |[draw,fill=light]| \textcolor{black}{9} &
  |[draw,fill=light]| \textcolor{black}{7} &
  |[draw,fill=white]| 3 \\
  |[draw,fill=light]| \textcolor{black}{7} &
  |[draw,fill=light]| \textcolor{black}{5} &
  |[draw,fill=white]| 0 &
  |[draw,fill=white]| 3 &
  |[draw,fill=light]| \textcolor{black}{4} &
  |[draw,fill=white]| 1 &
  |[draw,fill=light]| \textcolor{black}{9} &
  |[draw,fill=light]| \textcolor{black}{6} &
  |[draw,fill=white]| 2 &
  |[draw,fill=light]| \textcolor{black}{8} \\
  |[draw,fill=dark]| \textcolor{white}{8} &
  |[draw,fill=light]| \textcolor{black}{7} &
  |[draw,fill=white]| 2 &
  |[draw,fill=dark]| \textcolor{white}{5} &
  |[draw,fill=white]| 1 &
  |[draw,fill=light]| \textcolor{black}{9} &
  |[draw,fill=light]| \textcolor{black}{6} &
  |[draw,fill=white]| 3 &
  |[draw,fill=light]| \textcolor{black}{4} &
  |[draw,fill=white]| 0 \\
  |[draw,fill=light]| \textcolor{black}{9} &
  |[draw,fill=white]| 1 &
  |[draw,fill=light]| \textcolor{black}{4} &
  |[draw,fill=dark]| \textcolor{white}{8} &
  |[draw,fill=white]| 2 &
  |[draw,fill=dark]| \textcolor{white}{7} &
  |[draw,fill=white]| 0 &
  |[draw,fill=light]| \textcolor{black}{5} &
  |[draw,fill=white]| 3 &
  |[draw,fill=light]| \textcolor{black}{6} \\
  |[draw,fill=white]| 3 &
  |[draw,fill=dark]| \textcolor{white}{8} &
  |[draw,fill=dark]| \textcolor{white}{9} &
  |[draw,fill=light]| \textcolor{black}{7} &
  |[draw,fill=dark]| \textcolor{white}{6} &
  |[draw,fill=dark]| \textcolor{white}{4} &
  |[draw,fill=white]| 2 &
  |[draw,fill=white]| 0 &
  |[draw,fill=white]| 1 &
  |[draw,fill=light]| \textcolor{black}{5} \\
  |[draw,fill=dark]| \textcolor{white}{4} &
  |[draw,fill=dark]| \textcolor{white}{9} &
  |[draw,fill=dark]| \textcolor{white}{8} &
  |[draw,fill=light]| \textcolor{black}{6} &
  |[draw,fill=dark]| \textcolor{white}{7} &
  |[draw,fill=white]| 3 &
  |[draw,fill=light]| \textcolor{black}{5} &
  |[draw,fill=white]| 1 &
  |[draw,fill=white]| 0 &
  |[draw,fill=white]| 2 \\
  };\end{tikzpicture}
  \end{minipage}
  \hfill
  \begin{minipage}[c][1\width]{0.44\textwidth}
  \centering
  \begin{tikzpicture}
  \matrix (m) [matrix of nodes,nodes={element},column sep=-\pgflinewidth, row sep=-\pgflinewidth, label={[font=\large]above: type W}]{
  |[draw,fill=white]| 0 &
  |[draw,fill=white]| 3 &
  |[draw,fill=light]| \textcolor{black}{7} &
  |[draw,fill=white]| 2 &
  |[draw,fill=light]| \textcolor{black}{4} &
  |[draw,fill=light]| \textcolor{black}{9} &
  |[draw,fill=light]| \textcolor{black}{5} &
  |[draw,fill=light]| \textcolor{black}{8} &
  |[draw,fill=white]| 1 &
  |[draw,fill=light]| \textcolor{black}{6} \\
  |[draw,fill=white]| 1 &
  |[draw,fill=light]| \textcolor{black}{5} &
  |[draw,fill=white]| 3 &
  |[draw,fill=light]| \textcolor{black}{7} &
  |[draw,fill=white]| 2 &
  |[draw,fill=light]| \textcolor{black}{8} &
  |[draw,fill=light]| \textcolor{black}{6} &
  |[draw,fill=light]| \textcolor{black}{9} &
  |[draw,fill=white]| 0 &
  |[draw,fill=light]| \textcolor{black}{4} \\
  |[draw,fill=white]| 3 &
  |[draw,fill=light]| \textcolor{black}{7} &
  |[draw,fill=white]| 1 &
  |[draw,fill=light]| \textcolor{black}{6} &
  |[draw,fill=light]| \textcolor{black}{5} &
  |[draw,fill=white]| 2 &
  |[draw,fill=white]| 0 &
  |[draw,fill=light]| \textcolor{black}{4} &
  |[draw,fill=light]| \textcolor{black}{9} &
  |[draw,fill=light]| \textcolor{black}{8} \\
  |[draw,fill=light]| \textcolor{black}{5} &
  |[draw,fill=white]| 0 &
  |[draw,fill=light]| \textcolor{black}{4} &
  |[draw,fill=light]| \textcolor{black}{9} &
  |[draw,fill=white]| 1 &
  |[draw,fill=white]| 3 &
  |[draw,fill=white]| 2 &
  |[draw,fill=light]| \textcolor{black}{6} &
  |[draw,fill=light]| \textcolor{black}{8} &
  |[draw,fill=light]| \textcolor{black}{7} \\
  |[draw,fill=light]| \textcolor{black}{9} &
  |[draw,fill=white]| 2 &
  |[draw,fill=light]| \textcolor{black}{6} &
  |[draw,fill=white]| 3 &
  |[draw,fill=light]| \textcolor{black}{8} &
  |[draw,fill=white]| 0 &
  |[draw,fill=light]| \textcolor{black}{7} &
  |[draw,fill=white]| 1 &
  |[draw,fill=light]| \textcolor{black}{4} &
  |[draw,fill=light]| \textcolor{black}{5} \\
  |[draw,fill=white]| 2 &
  |[draw,fill=light]| \textcolor{black}{4} &
  |[draw,fill=dark]| \textcolor{white}{8} &
  |[draw,fill=white]| 1 &
  |[draw,fill=dark]| \textcolor{white}{6} &
  |[draw,fill=light]| \textcolor{black}{5} &
  |[draw,fill=light]| \textcolor{black}{9} &
  |[draw,fill=light]| \textcolor{black}{7} &
  |[draw,fill=white]| 3 &
  |[draw,fill=white]| 0 \\
  |[draw,fill=light]| \textcolor{black}{6} &
  |[draw,fill=white]| 1 &
  |[draw,fill=light]| \textcolor{black}{5} &
  |[draw,fill=white]| 0 &
  |[draw,fill=dark]| \textcolor{white}{7} &
  |[draw,fill=dark]| \textcolor{white}{4} &
  |[draw,fill=light]| \textcolor{black}{8} &
  |[draw,fill=white]| 3 &
  |[draw,fill=white]| 2 &
  |[draw,fill=light]| \textcolor{black}{9} \\
  |[draw,fill=light]| \textcolor{black}{7} &
  |[draw,fill=dark]| \textcolor{white}{9} &
  |[draw,fill=white]| 2 &
  |[draw,fill=dark]| \textcolor{white}{8} &
  |[draw,fill=white]| 3 &
  |[draw,fill=light]| \textcolor{black}{6} &
  |[draw,fill=light]| \textcolor{black}{4} &
  |[draw,fill=white]| 0 &
  |[draw,fill=light]| \textcolor{black}{5} &
  |[draw,fill=white]| 1 \\
  |[draw,fill=dark]| \textcolor{white}{8} &
  |[draw,fill=light]| \textcolor{black}{6} &
  |[draw,fill=dark]| \textcolor{white}{9} &
  |[draw,fill=light]| \textcolor{black}{4} &
  |[draw,fill=white]| 0 &
  |[draw,fill=white]| 1 &
  |[draw,fill=white]| 3 &
  |[draw,fill=light]| \textcolor{black}{5} &
  |[draw,fill=light]| \textcolor{black}{7} &
  |[draw,fill=white]| 2 \\
  |[draw,fill=dark]| \textcolor{white}{4} &
  |[draw,fill=dark]| \textcolor{white}{8} &
  |[draw,fill=white]| 0 &
  |[draw,fill=dark]| \textcolor{white}{5} &
  |[draw,fill=light]| \textcolor{black}{9} &
  |[draw,fill=dark]| \textcolor{white}{7} &
  |[draw,fill=white]| 1 &
  |[draw,fill=white]| 2 &
  |[draw,fill=light]| \textcolor{black}{6} &
  |[draw,fill=white]| 3 \\
  };\end{tikzpicture}
  \end{minipage}
\end{figure}
\begin{figure}[ht]
  \begin{minipage}[c][1\width]{0.44\textwidth}
  \centering
  \begin{tikzpicture}
  \matrix (m) [matrix of nodes,nodes={element},column sep=-\pgflinewidth, row sep=-\pgflinewidth, label={[font=\large]above: type U}]{
  |[draw,fill=white]| 0 &
  |[draw,fill=white]| 1 &
  |[draw,fill=white]| 2 &
  |[draw,fill=light]| \textcolor{black}{4} &
  |[draw,fill=light]| \textcolor{black}{5} &
  |[draw,fill=light]| \textcolor{black}{6} &
  |[draw,fill=white]| 3 &
  |[draw,fill=light]| \textcolor{black}{7} &
  |[draw,fill=light]| \textcolor{black}{8} &
  |[draw,fill=light]| \textcolor{black}{9} \\
  |[draw,fill=white]| 1 &
  |[draw,fill=white]| 3 &
  |[draw,fill=light]| \textcolor{black}{8} &
  |[draw,fill=white]| 2 &
  |[draw,fill=light]| \textcolor{black}{6} &
  |[draw,fill=light]| \textcolor{black}{4} &
  |[draw,fill=light]| \textcolor{black}{7} &
  |[draw,fill=light]| \textcolor{black}{5} &
  |[draw,fill=light]| \textcolor{black}{9} &
  |[draw,fill=white]| 0 \\
  |[draw,fill=white]| 2 &
  |[draw,fill=light]| \textcolor{black}{8} &
  |[draw,fill=light]| \textcolor{black}{9} &
  |[draw,fill=light]| \textcolor{black}{5} &
  |[draw,fill=white]| 3 &
  |[draw,fill=white]| 1 &
  |[draw,fill=light]| \textcolor{black}{4} &
  |[draw,fill=white]| 0 &
  |[draw,fill=light]| \textcolor{black}{7} &
  |[draw,fill=light]| \textcolor{black}{6} \\
  |[draw,fill=light]| \textcolor{black}{5} &
  |[draw,fill=light]| \textcolor{black}{9} &
  |[draw,fill=white]| 0 &
  |[draw,fill=white]| 1 &
  |[draw,fill=white]| 2 &
  |[draw,fill=light]| \textcolor{black}{7} &
  |[draw,fill=light]| \textcolor{black}{8} &
  |[draw,fill=light]| \textcolor{black}{6} &
  |[draw,fill=white]| 3 &
  |[draw,fill=light]| \textcolor{black}{4} \\
  |[draw,fill=light]| \textcolor{black}{7} &
  |[draw,fill=white]| 0 &
  |[draw,fill=light]| \textcolor{black}{4} &
  |[draw,fill=white]| 3 &
  |[draw,fill=white]| 1 &
  |[draw,fill=light]| \textcolor{black}{9} &
  |[draw,fill=white]| 2 &
  |[draw,fill=light]| \textcolor{black}{8} &
  |[draw,fill=light]| \textcolor{black}{6} &
  |[draw,fill=light]| \textcolor{black}{5} \\
  |[draw,fill=light]| \textcolor{black}{9} &
  |[draw,fill=light]| \textcolor{black}{5} &
  |[draw,fill=white]| 3 &
  |[draw,fill=light]| \textcolor{black}{7} &
  |[draw,fill=white]| 0 &
  |[draw,fill=white]| 2 &
  |[draw,fill=light]| \textcolor{black}{6} &
  |[draw,fill=white]| 1 &
  |[draw,fill=light]| \textcolor{black}{4} &
  |[draw,fill=light]| \textcolor{black}{8} \\
  |[draw,fill=white]| 3 &
  |[draw,fill=light]| \textcolor{black}{4} &
  |[draw,fill=dark]| \textcolor{white}{7} &
  |[draw,fill=dark]| \textcolor{white}{6} &
  |[draw,fill=light]| \textcolor{black}{8} &
  |[draw,fill=white]| 0 &
  |[draw,fill=light]| \textcolor{black}{5} &
  |[draw,fill=light]| \textcolor{black}{9} &
  |[draw,fill=white]| 2 &
  |[draw,fill=white]| 1 \\
  |[draw,fill=light]| \textcolor{black}{6} &
  |[draw,fill=dark]| \textcolor{white}{7} &
  |[draw,fill=white]| 1 &
  |[draw,fill=white]| 0 &
  |[draw,fill=light]| \textcolor{black}{4} &
  |[draw,fill=dark]| \textcolor{white}{8} &
  |[draw,fill=light]| \textcolor{black}{9} &
  |[draw,fill=white]| 3 &
  |[draw,fill=light]| \textcolor{black}{5} &
  |[draw,fill=white]| 2 \\
  |[draw,fill=dark]| \textcolor{white}{4} &
  |[draw,fill=dark]| \textcolor{white}{6} &
  |[draw,fill=dark]| \textcolor{white}{5} &
  |[draw,fill=light]| \textcolor{black}{8} &
  |[draw,fill=dark]| \textcolor{white}{9} &
  |[draw,fill=white]| 3 &
  |[draw,fill=white]| 0 &
  |[draw,fill=white]| 2 &
  |[draw,fill=white]| 1 &
  |[draw,fill=light]| \textcolor{black}{7} \\
  |[draw,fill=dark]| \textcolor{white}{8} &
  |[draw,fill=white]| 2 &
  |[draw,fill=light]| \textcolor{black}{6} &
  |[draw,fill=dark]| \textcolor{white}{9} &
  |[draw,fill=dark]| \textcolor{white}{7} &
  |[draw,fill=dark]| \textcolor{white}{5} &
  |[draw,fill=white]| 1 &
  |[draw,fill=light]| \textcolor{black}{4} &
  |[draw,fill=white]| 0 &
  |[draw,fill=white]| 3 \\
  };\end{tikzpicture}
  \end{minipage}
  \hfill
  \begin{minipage}[c][1\width]{0.44\textwidth}
  \centering
  \begin{tikzpicture}
  \matrix (m) [matrix of nodes,nodes={element},column sep=-\pgflinewidth, row sep=-\pgflinewidth, label={[font=\large]above: type X}]{
  |[draw,fill=white]| 0 &
  |[draw,fill=light]| \textcolor{black}{5} &
  |[draw,fill=light]| \textcolor{black}{9} &
  |[draw,fill=white]| 1 &
  |[draw,fill=light]| \textcolor{black}{8} &
  |[draw,fill=white]| 3 &
  |[draw,fill=light]| \textcolor{black}{7} &
  |[draw,fill=light]| \textcolor{black}{4} &
  |[draw,fill=light]| \textcolor{black}{6} &
  |[draw,fill=white]| 2 \\
  |[draw,fill=white]| 3 &
  |[draw,fill=white]| 2 &
  |[draw,fill=white]| 1 &
  |[draw,fill=light]| \textcolor{black}{5} &
  |[draw,fill=light]| \textcolor{black}{6} &
  |[draw,fill=light]| \textcolor{black}{7} &
  |[draw,fill=white]| 0 &
  |[draw,fill=light]| \textcolor{black}{8} &
  |[draw,fill=light]| \textcolor{black}{4} &
  |[draw,fill=light]| \textcolor{black}{9} \\
  |[draw,fill=light]| \textcolor{black}{5} &
  |[draw,fill=white]| 3 &
  |[draw,fill=light]| \textcolor{black}{6} &
  |[draw,fill=light]| \textcolor{black}{4} &
  |[draw,fill=white]| 1 &
  |[draw,fill=white]| 0 &
  |[draw,fill=light]| \textcolor{black}{9} &
  |[draw,fill=white]| 2 &
  |[draw,fill=light]| \textcolor{black}{7} &
  |[draw,fill=light]| \textcolor{black}{8} \\
  |[draw,fill=light]| \textcolor{black}{6} &
  |[draw,fill=light]| \textcolor{black}{4} &
  |[draw,fill=white]| 0 &
  |[draw,fill=light]| \textcolor{black}{8} &
  |[draw,fill=white]| 3 &
  |[draw,fill=white]| 2 &
  |[draw,fill=white]| 1 &
  |[draw,fill=light]| \textcolor{black}{7} &
  |[draw,fill=light]| \textcolor{black}{9} &
  |[draw,fill=light]| \textcolor{black}{5} \\
  |[draw,fill=white]| 1 &
  |[draw,fill=light]| \textcolor{black}{9} &
  |[draw,fill=white]| 3 &
  |[draw,fill=dark]| \textcolor{white}{6} &
  |[draw,fill=light]| \textcolor{black}{4} &
  |[draw,fill=dark]| \textcolor{white}{5} &
  |[draw,fill=white]| 2 &
  |[draw,fill=white]| 0 &
  |[draw,fill=light]| \textcolor{black}{8} &
  |[draw,fill=light]| \textcolor{black}{7} \\
  |[draw,fill=white]| 2 &
  |[draw,fill=dark]| \textcolor{white}{7} &
  |[draw,fill=light]| \textcolor{black}{8} &
  |[draw,fill=white]| 3 &
  |[draw,fill=dark]| \textcolor{white}{9} &
  |[draw,fill=light]| \textcolor{black}{6} &
  |[draw,fill=light]| \textcolor{black}{5} &
  |[draw,fill=white]| 1 &
  |[draw,fill=white]| 0 &
  |[draw,fill=light]| \textcolor{black}{4} \\
  |[draw,fill=dark]| \textcolor{white}{4} &
  |[draw,fill=light]| \textcolor{black}{8} &
  |[draw,fill=white]| 2 &
  |[draw,fill=white]| 0 &
  |[draw,fill=dark]| \textcolor{white}{7} &
  |[draw,fill=light]| \textcolor{black}{9} &
  |[draw,fill=light]| \textcolor{black}{6} &
  |[draw,fill=light]| \textcolor{black}{5} &
  |[draw,fill=white]| 3 &
  |[draw,fill=white]| 1 \\
  |[draw,fill=light]| \textcolor{black}{7} &
  |[draw,fill=white]| 1 &
  |[draw,fill=dark]| \textcolor{white}{5} &
  |[draw,fill=dark]| \textcolor{white}{9} &
  |[draw,fill=white]| 0 &
  |[draw,fill=light]| \textcolor{black}{4} &
  |[draw,fill=light]| \textcolor{black}{8} &
  |[draw,fill=white]| 3 &
  |[draw,fill=white]| 2 &
  |[draw,fill=light]| \textcolor{black}{6} \\
  |[draw,fill=dark]| \textcolor{white}{8} &
  |[draw,fill=dark]| \textcolor{white}{6} &
  |[draw,fill=light]| \textcolor{black}{4} &
  |[draw,fill=light]| \textcolor{black}{7} &
  |[draw,fill=white]| 2 &
  |[draw,fill=white]| 1 &
  |[draw,fill=white]| 3 &
  |[draw,fill=light]| \textcolor{black}{9} &
  |[draw,fill=light]| \textcolor{black}{5} &
  |[draw,fill=white]| 0 \\
  |[draw,fill=light]| \textcolor{black}{9} &
  |[draw,fill=white]| 0 &
  |[draw,fill=dark]| \textcolor{white}{7} &
  |[draw,fill=white]| 2 &
  |[draw,fill=light]| \textcolor{black}{5} &
  |[draw,fill=dark]| \textcolor{white}{8} &
  |[draw,fill=light]| \textcolor{black}{4} &
  |[draw,fill=light]| \textcolor{black}{6} &
  |[draw,fill=white]| 1 &
  |[draw,fill=white]| 3 \\
  };\end{tikzpicture}
  \end{minipage}
\end{figure}
\begin{figure}[ht]
  \begin{minipage}[c][1\width]{0.44\textwidth}
  \centering
  \begin{tikzpicture}
  \matrix (m) [matrix of nodes,nodes={element},column sep=-\pgflinewidth, row sep=-\pgflinewidth, label={[font=\large]above: type V}]{
  |[draw,fill=white]| 0 &
  |[draw,fill=white]| 2 &
  |[draw,fill=white]| 3 &
  |[draw,fill=light]| \textcolor{black}{4} &
  |[draw,fill=light]| \textcolor{black}{5} &
  |[draw,fill=light]| \textcolor{black}{6} &
  |[draw,fill=white]| 1 &
  |[draw,fill=light]| \textcolor{black}{7} &
  |[draw,fill=light]| \textcolor{black}{8} &
  |[draw,fill=light]| \textcolor{black}{9} \\
  |[draw,fill=white]| 1 &
  |[draw,fill=white]| 3 &
  |[draw,fill=light]| \textcolor{black}{7} &
  |[draw,fill=light]| \textcolor{black}{5} &
  |[draw,fill=white]| 0 &
  |[draw,fill=light]| \textcolor{black}{9} &
  |[draw,fill=light]| \textcolor{black}{4} &
  |[draw,fill=light]| \textcolor{black}{8} &
  |[draw,fill=white]| 2 &
  |[draw,fill=light]| \textcolor{black}{6} \\
  |[draw,fill=white]| 3 &
  |[draw,fill=white]| 0 &
  |[draw,fill=white]| 2 &
  |[draw,fill=light]| \textcolor{black}{9} &
  |[draw,fill=light]| \textcolor{black}{6} &
  |[draw,fill=light]| \textcolor{black}{8} &
  |[draw,fill=light]| \textcolor{black}{5} &
  |[draw,fill=white]| 1 &
  |[draw,fill=light]| \textcolor{black}{4} &
  |[draw,fill=light]| \textcolor{black}{7} \\
  |[draw,fill=light]| \textcolor{black}{4} &
  |[draw,fill=light]| \textcolor{black}{6} &
  |[draw,fill=white]| 1 &
  |[draw,fill=light]| \textcolor{black}{8} &
  |[draw,fill=white]| 2 &
  |[draw,fill=white]| 0 &
  |[draw,fill=light]| \textcolor{black}{7} &
  |[draw,fill=light]| \textcolor{black}{9} &
  |[draw,fill=light]| \textcolor{black}{5} &
  |[draw,fill=white]| 3 \\
  |[draw,fill=light]| \textcolor{black}{5} &
  |[draw,fill=white]| 1 &
  |[draw,fill=white]| 0 &
  |[draw,fill=white]| 3 &
  |[draw,fill=light]| \textcolor{black}{9} &
  |[draw,fill=light]| \textcolor{black}{7} &
  |[draw,fill=white]| 2 &
  |[draw,fill=light]| \textcolor{black}{4} &
  |[draw,fill=light]| \textcolor{black}{6} &
  |[draw,fill=light]| \textcolor{black}{8} \\
  |[draw,fill=light]| \textcolor{black}{7} &
  |[draw,fill=light]| \textcolor{black}{5} &
  |[draw,fill=light]| \textcolor{black}{4} &
  |[draw,fill=white]| 0 &
  |[draw,fill=white]| 1 &
  |[draw,fill=white]| 3 &
  |[draw,fill=light]| \textcolor{black}{8} &
  |[draw,fill=light]| \textcolor{black}{6} &
  |[draw,fill=light]| \textcolor{black}{9} &
  |[draw,fill=white]| 2 \\
  |[draw,fill=white]| 2 &
  |[draw,fill=light]| \textcolor{black}{7} &
  |[draw,fill=light]| \textcolor{black}{8} &
  |[draw,fill=dark]| \textcolor{white}{6} &
  |[draw,fill=white]| 3 &
  |[draw,fill=dark]| \textcolor{white}{4} &
  |[draw,fill=light]| \textcolor{black}{9} &
  |[draw,fill=white]| 0 &
  |[draw,fill=white]| 1 &
  |[draw,fill=light]| \textcolor{black}{5} \\
  |[draw,fill=light]| \textcolor{black}{6} &
  |[draw,fill=light]| \textcolor{black}{8} &
  |[draw,fill=dark]| \textcolor{white}{9} &
  |[draw,fill=white]| 2 &
  |[draw,fill=dark]| \textcolor{white}{7} &
  |[draw,fill=white]| 1 &
  |[draw,fill=white]| 3 &
  |[draw,fill=light]| \textcolor{black}{5} &
  |[draw,fill=white]| 0 &
  |[draw,fill=light]| \textcolor{black}{4} \\
  |[draw,fill=dark]| \textcolor{white}{9} &
  |[draw,fill=dark]| \textcolor{white}{4} &
  |[draw,fill=light]| \textcolor{black}{5} &
  |[draw,fill=white]| 1 &
  |[draw,fill=light]| \textcolor{black}{8} &
  |[draw,fill=white]| 2 &
  |[draw,fill=light]| \textcolor{black}{6} &
  |[draw,fill=white]| 3 &
  |[draw,fill=light]| \textcolor{black}{7} &
  |[draw,fill=white]| 0 \\
  |[draw,fill=dark]| \textcolor{white}{8} &
  |[draw,fill=dark]| \textcolor{white}{9} &
  |[draw,fill=dark]| \textcolor{white}{6} &
  |[draw,fill=dark]| \textcolor{white}{7} &
  |[draw,fill=dark]| \textcolor{white}{4} &
  |[draw,fill=dark]| \textcolor{white}{5} &
  |[draw,fill=white]| 0 &
  |[draw,fill=white]| 2 &
  |[draw,fill=white]| 3 &
  |[draw,fill=white]| 1 \\
  };\end{tikzpicture}
  \end{minipage}
  \hfill
  \begin{minipage}[c][1\width]{0.44\textwidth}
  \centering
  \begin{tikzpicture}
  \matrix (m) [matrix of nodes,nodes={element},column sep=-\pgflinewidth, row sep=-\pgflinewidth, label={[font=\large]above: type X}]{
  |[draw,fill=white]| 0 &
  |[draw,fill=white]| 3 &
  |[draw,fill=light]| \textcolor{black}{4} &
  |[draw,fill=light]| \textcolor{black}{9} &
  |[draw,fill=light]| \textcolor{black}{8} &
  |[draw,fill=white]| 1 &
  |[draw,fill=light]| \textcolor{black}{7} &
  |[draw,fill=white]| 2 &
  |[draw,fill=light]| \textcolor{black}{6} &
  |[draw,fill=light]| \textcolor{black}{5} \\
  |[draw,fill=white]| 1 &
  |[draw,fill=light]| \textcolor{black}{5} &
  |[draw,fill=white]| 0 &
  |[draw,fill=light]| \textcolor{black}{8} &
  |[draw,fill=light]| \textcolor{black}{6} &
  |[draw,fill=white]| 2 &
  |[draw,fill=light]| \textcolor{black}{9} &
  |[draw,fill=light]| \textcolor{black}{7} &
  |[draw,fill=white]| 3 &
  |[draw,fill=light]| \textcolor{black}{4} \\
  |[draw,fill=white]| 2 &
  |[draw,fill=light]| \textcolor{black}{6} &
  |[draw,fill=white]| 3 &
  |[draw,fill=white]| 0 &
  |[draw,fill=light]| \textcolor{black}{9} &
  |[draw,fill=light]| \textcolor{black}{8} &
  |[draw,fill=light]| \textcolor{black}{4} &
  |[draw,fill=light]| \textcolor{black}{5} &
  |[draw,fill=light]| \textcolor{black}{7} &
  |[draw,fill=white]| 1 \\
  |[draw,fill=light]| \textcolor{black}{4} &
  |[draw,fill=white]| 1 &
  |[draw,fill=light]| \textcolor{black}{5} &
  |[draw,fill=white]| 2 &
  |[draw,fill=white]| 3 &
  |[draw,fill=light]| \textcolor{black}{9} &
  |[draw,fill=white]| 0 &
  |[draw,fill=light]| \textcolor{black}{6} &
  |[draw,fill=light]| \textcolor{black}{8} &
  |[draw,fill=light]| \textcolor{black}{7} \\
  |[draw,fill=white]| 3 &
  |[draw,fill=light]| \textcolor{black}{8} &
  |[draw,fill=light]| \textcolor{black}{7} &
  |[draw,fill=dark]| \textcolor{white}{6} &
  |[draw,fill=white]| 2 &
  |[draw,fill=dark]| \textcolor{white}{5} &
  |[draw,fill=white]| 1 &
  |[draw,fill=light]| \textcolor{black}{4} &
  |[draw,fill=light]| \textcolor{black}{9} &
  |[draw,fill=white]| 0 \\
  |[draw,fill=light]| \textcolor{black}{5} &
  |[draw,fill=dark]| \textcolor{white}{9} &
  |[draw,fill=white]| 2 &
  |[draw,fill=light]| \textcolor{black}{4} &
  |[draw,fill=dark]| \textcolor{white}{7} &
  |[draw,fill=white]| 0 &
  |[draw,fill=light]| \textcolor{black}{8} &
  |[draw,fill=white]| 3 &
  |[draw,fill=white]| 1 &
  |[draw,fill=light]| \textcolor{black}{6} \\
  |[draw,fill=light]| \textcolor{black}{6} &
  |[draw,fill=dark]| \textcolor{white}{4} &
  |[draw,fill=light]| \textcolor{black}{8} &
  |[draw,fill=dark]| \textcolor{white}{7} &
  |[draw,fill=white]| 0 &
  |[draw,fill=white]| 3 &
  |[draw,fill=white]| 2 &
  |[draw,fill=white]| 1 &
  |[draw,fill=light]| \textcolor{black}{5} &
  |[draw,fill=light]| \textcolor{black}{9} \\
  |[draw,fill=light]| \textcolor{black}{7} &
  |[draw,fill=white]| 0 &
  |[draw,fill=dark]| \textcolor{white}{6} &
  |[draw,fill=white]| 1 &
  |[draw,fill=light]| \textcolor{black}{5} &
  |[draw,fill=dark]| \textcolor{white}{4} &
  |[draw,fill=white]| 3 &
  |[draw,fill=light]| \textcolor{black}{9} &
  |[draw,fill=white]| 2 &
  |[draw,fill=light]| \textcolor{black}{8} \\
  |[draw,fill=dark]| \textcolor{white}{8} &
  |[draw,fill=white]| 2 &
  |[draw,fill=dark]| \textcolor{white}{9} &
  |[draw,fill=light]| \textcolor{black}{5} &
  |[draw,fill=white]| 1 &
  |[draw,fill=light]| \textcolor{black}{7} &
  |[draw,fill=light]| \textcolor{black}{6} &
  |[draw,fill=white]| 0 &
  |[draw,fill=light]| \textcolor{black}{4} &
  |[draw,fill=white]| 3 \\
  |[draw,fill=dark]| \textcolor{white}{9} &
  |[draw,fill=light]| \textcolor{black}{7} &
  |[draw,fill=white]| 1 &
  |[draw,fill=white]| 3 &
  |[draw,fill=dark]| \textcolor{white}{4} &
  |[draw,fill=light]| \textcolor{black}{6} &
  |[draw,fill=light]| \textcolor{black}{5} &
  |[draw,fill=light]| \textcolor{black}{8} &
  |[draw,fill=white]| 0 &
  |[draw,fill=white]| 2 \\
  };\end{tikzpicture}
  \end{minipage}
\end{figure}
\begin{figure}[ht]
  \begin{minipage}[c][1\width]{0.44\textwidth}
  \centering
  \begin{tikzpicture}
  \matrix (m) [matrix of nodes,nodes={element},column sep=-\pgflinewidth, row sep=-\pgflinewidth, label={[font=\large]above: type W}]{
  |[draw,fill=white]| 0 &
  |[draw,fill=white]| 2 &
  |[draw,fill=white]| 3 &
  |[draw,fill=light]| \textcolor{black}{4} &
  |[draw,fill=light]| \textcolor{black}{5} &
  |[draw,fill=light]| \textcolor{black}{6} &
  |[draw,fill=white]| 1 &
  |[draw,fill=light]| \textcolor{black}{7} &
  |[draw,fill=light]| \textcolor{black}{8} &
  |[draw,fill=light]| \textcolor{black}{9} \\
  |[draw,fill=white]| 1 &
  |[draw,fill=white]| 3 &
  |[draw,fill=light]| \textcolor{black}{7} &
  |[draw,fill=light]| \textcolor{black}{6} &
  |[draw,fill=white]| 2 &
  |[draw,fill=light]| \textcolor{black}{9} &
  |[draw,fill=light]| \textcolor{black}{4} &
  |[draw,fill=white]| 0 &
  |[draw,fill=light]| \textcolor{black}{5} &
  |[draw,fill=light]| \textcolor{black}{8} \\
  |[draw,fill=white]| 2 &
  |[draw,fill=light]| \textcolor{black}{7} &
  |[draw,fill=white]| 1 &
  |[draw,fill=light]| \textcolor{black}{5} &
  |[draw,fill=white]| 0 &
  |[draw,fill=light]| \textcolor{black}{4} &
  |[draw,fill=light]| \textcolor{black}{8} &
  |[draw,fill=light]| \textcolor{black}{9} &
  |[draw,fill=white]| 3 &
  |[draw,fill=light]| \textcolor{black}{6} \\
  |[draw,fill=white]| 3 &
  |[draw,fill=light]| \textcolor{black}{5} &
  |[draw,fill=light]| \textcolor{black}{4} &
  |[draw,fill=white]| 0 &
  |[draw,fill=light]| \textcolor{black}{8} &
  |[draw,fill=white]| 1 &
  |[draw,fill=light]| \textcolor{black}{7} &
  |[draw,fill=light]| \textcolor{black}{6} &
  |[draw,fill=light]| \textcolor{black}{9} &
  |[draw,fill=white]| 2 \\
  |[draw,fill=light]| \textcolor{black}{8} &
  |[draw,fill=white]| 0 &
  |[draw,fill=white]| 2 &
  |[draw,fill=white]| 3 &
  |[draw,fill=light]| \textcolor{black}{9} &
  |[draw,fill=light]| \textcolor{black}{7} &
  |[draw,fill=light]| \textcolor{black}{5} &
  |[draw,fill=white]| 1 &
  |[draw,fill=light]| \textcolor{black}{6} &
  |[draw,fill=light]| \textcolor{black}{4} \\
  |[draw,fill=dark]| \textcolor{white}{4} &
  |[draw,fill=white]| 1 &
  |[draw,fill=light]| \textcolor{black}{6} &
  |[draw,fill=light]| \textcolor{black}{8} &
  |[draw,fill=white]| 3 &
  |[draw,fill=dark]| \textcolor{white}{5} &
  |[draw,fill=light]| \textcolor{black}{9} &
  |[draw,fill=white]| 2 &
  |[draw,fill=white]| 0 &
  |[draw,fill=light]| \textcolor{black}{7} \\
  |[draw,fill=light]| \textcolor{black}{5} &
  |[draw,fill=light]| \textcolor{black}{8} &
  |[draw,fill=dark]| \textcolor{white}{9} &
  |[draw,fill=dark]| \textcolor{white}{7} &
  |[draw,fill=white]| 1 &
  |[draw,fill=white]| 2 &
  |[draw,fill=light]| \textcolor{black}{6} &
  |[draw,fill=white]| 3 &
  |[draw,fill=light]| \textcolor{black}{4} &
  |[draw,fill=white]| 0 \\
  |[draw,fill=light]| \textcolor{black}{6} &
  |[draw,fill=light]| \textcolor{black}{9} &
  |[draw,fill=white]| 0 &
  |[draw,fill=white]| 1 &
  |[draw,fill=dark]| \textcolor{white}{7} &
  |[draw,fill=dark]| \textcolor{white}{8} &
  |[draw,fill=white]| 3 &
  |[draw,fill=light]| \textcolor{black}{4} &
  |[draw,fill=white]| 2 &
  |[draw,fill=light]| \textcolor{black}{5} \\
  |[draw,fill=light]| \textcolor{black}{9} &
  |[draw,fill=dark]| \textcolor{white}{4} &
  |[draw,fill=light]| \textcolor{black}{5} &
  |[draw,fill=white]| 2 &
  |[draw,fill=dark]| \textcolor{white}{6} &
  |[draw,fill=white]| 3 &
  |[draw,fill=white]| 0 &
  |[draw,fill=light]| \textcolor{black}{8} &
  |[draw,fill=light]| \textcolor{black}{7} &
  |[draw,fill=white]| 1 \\
  |[draw,fill=dark]| \textcolor{white}{7} &
  |[draw,fill=dark]| \textcolor{white}{6} &
  |[draw,fill=dark]| \textcolor{white}{8} &
  |[draw,fill=dark]| \textcolor{white}{9} &
  |[draw,fill=light]| \textcolor{black}{4} &
  |[draw,fill=white]| 0 &
  |[draw,fill=white]| 2 &
  |[draw,fill=light]| \textcolor{black}{5} &
  |[draw,fill=white]| 1 &
  |[draw,fill=white]| 3 \\
  };\end{tikzpicture}
  \end{minipage}
  \hfill
  \begin{minipage}[c][1\width]{0.44\textwidth}
  \centering
  \begin{tikzpicture}
  \matrix (m) [matrix of nodes,nodes={element},column sep=-\pgflinewidth, row sep=-\pgflinewidth, label={[font=\large]above: type W}]{
  |[draw,fill=white]| 0 &
  |[draw,fill=white]| 3 &
  |[draw,fill=light]| \textcolor{black}{5} &
  |[draw,fill=light]| \textcolor{black}{8} &
  |[draw,fill=white]| 1 &
  |[draw,fill=light]| \textcolor{black}{7} &
  |[draw,fill=white]| 2 &
  |[draw,fill=light]| \textcolor{black}{4} &
  |[draw,fill=light]| \textcolor{black}{9} &
  |[draw,fill=light]| \textcolor{black}{6} \\
  |[draw,fill=white]| 2 &
  |[draw,fill=white]| 0 &
  |[draw,fill=light]| \textcolor{black}{6} &
  |[draw,fill=white]| 1 &
  |[draw,fill=light]| \textcolor{black}{5} &
  |[draw,fill=light]| \textcolor{black}{9} &
  |[draw,fill=light]| \textcolor{black}{7} &
  |[draw,fill=light]| \textcolor{black}{8} &
  |[draw,fill=light]| \textcolor{black}{4} &
  |[draw,fill=white]| 3 \\
  |[draw,fill=light]| \textcolor{black}{5} &
  |[draw,fill=light]| \textcolor{black}{9} &
  |[draw,fill=white]| 3 &
  |[draw,fill=white]| 2 &
  |[draw,fill=light]| \textcolor{black}{4} &
  |[draw,fill=white]| 1 &
  |[draw,fill=light]| \textcolor{black}{8} &
  |[draw,fill=white]| 0 &
  |[draw,fill=light]| \textcolor{black}{6} &
  |[draw,fill=light]| \textcolor{black}{7} \\
  |[draw,fill=light]| \textcolor{black}{6} &
  |[draw,fill=white]| 1 &
  |[draw,fill=white]| 2 &
  |[draw,fill=light]| \textcolor{black}{5} &
  |[draw,fill=light]| \textcolor{black}{8} &
  |[draw,fill=white]| 0 &
  |[draw,fill=light]| \textcolor{black}{4} &
  |[draw,fill=white]| 3 &
  |[draw,fill=light]| \textcolor{black}{7} &
  |[draw,fill=light]| \textcolor{black}{9} \\
  |[draw,fill=light]| \textcolor{black}{9} &
  |[draw,fill=light]| \textcolor{black}{7} &
  |[draw,fill=white]| 0 &
  |[draw,fill=light]| \textcolor{black}{4} &
  |[draw,fill=white]| 3 &
  |[draw,fill=white]| 2 &
  |[draw,fill=light]| \textcolor{black}{5} &
  |[draw,fill=light]| \textcolor{black}{6} &
  |[draw,fill=white]| 1 &
  |[draw,fill=light]| \textcolor{black}{8} \\
  |[draw,fill=white]| 1 &
  |[draw,fill=light]| \textcolor{black}{5} &
  |[draw,fill=dark]| \textcolor{white}{8} &
  |[draw,fill=white]| 3 &
  |[draw,fill=dark]| \textcolor{white}{6} &
  |[draw,fill=light]| \textcolor{black}{4} &
  |[draw,fill=light]| \textcolor{black}{9} &
  |[draw,fill=light]| \textcolor{black}{7} &
  |[draw,fill=white]| 2 &
  |[draw,fill=white]| 0 \\
  |[draw,fill=white]| 3 &
  |[draw,fill=light]| \textcolor{black}{8} &
  |[draw,fill=white]| 1 &
  |[draw,fill=dark]| \textcolor{white}{9} &
  |[draw,fill=dark]| \textcolor{white}{7} &
  |[draw,fill=light]| \textcolor{black}{6} &
  |[draw,fill=white]| 0 &
  |[draw,fill=white]| 2 &
  |[draw,fill=light]| \textcolor{black}{5} &
  |[draw,fill=light]| \textcolor{black}{4} \\
  |[draw,fill=dark]| \textcolor{white}{4} &
  |[draw,fill=white]| 2 &
  |[draw,fill=light]| \textcolor{black}{7} &
  |[draw,fill=white]| 0 &
  |[draw,fill=light]| \textcolor{black}{9} &
  |[draw,fill=dark]| \textcolor{white}{8} &
  |[draw,fill=light]| \textcolor{black}{6} &
  |[draw,fill=light]| \textcolor{black}{5} &
  |[draw,fill=white]| 3 &
  |[draw,fill=white]| 1 \\
  |[draw,fill=light]| \textcolor{black}{8} &
  |[draw,fill=dark]| \textcolor{white}{6} &
  |[draw,fill=light]| \textcolor{black}{4} &
  |[draw,fill=dark]| \textcolor{white}{7} &
  |[draw,fill=white]| 2 &
  |[draw,fill=white]| 3 &
  |[draw,fill=white]| 1 &
  |[draw,fill=light]| \textcolor{black}{9} &
  |[draw,fill=white]| 0 &
  |[draw,fill=light]| \textcolor{black}{5} \\
  |[draw,fill=dark]| \textcolor{white}{7} &
  |[draw,fill=dark]| \textcolor{white}{4} &
  |[draw,fill=dark]| \textcolor{white}{9} &
  |[draw,fill=light]| \textcolor{black}{6} &
  |[draw,fill=white]| 0 &
  |[draw,fill=dark]| \textcolor{white}{5} &
  |[draw,fill=white]| 3 &
  |[draw,fill=white]| 1 &
  |[draw,fill=light]| \textcolor{black}{8} &
  |[draw,fill=white]| 2 \\
  };\end{tikzpicture}
  \end{minipage}
\end{figure}
\begin{figure}[ht]
  \begin{minipage}[c][1\width]{0.44\textwidth}
  \centering
  \begin{tikzpicture}
  \matrix (m) [matrix of nodes,nodes={element},column sep=-\pgflinewidth, row sep=-\pgflinewidth, label={[font=\large]above: type W}]{
  |[draw,fill=white]| 0 &
  |[draw,fill=white]| 1 &
  |[draw,fill=white]| 3 &
  |[draw,fill=light]| \textcolor{black}{4} &
  |[draw,fill=light]| \textcolor{black}{5} &
  |[draw,fill=light]| \textcolor{black}{6} &
  |[draw,fill=white]| 2 &
  |[draw,fill=light]| \textcolor{black}{7} &
  |[draw,fill=light]| \textcolor{black}{8} &
  |[draw,fill=light]| \textcolor{black}{9} \\
  |[draw,fill=white]| 1 &
  |[draw,fill=light]| \textcolor{black}{5} &
  |[draw,fill=white]| 0 &
  |[draw,fill=white]| 2 &
  |[draw,fill=light]| \textcolor{black}{7} &
  |[draw,fill=light]| \textcolor{black}{9} &
  |[draw,fill=light]| \textcolor{black}{8} &
  |[draw,fill=light]| \textcolor{black}{6} &
  |[draw,fill=light]| \textcolor{black}{4} &
  |[draw,fill=white]| 3 \\
  |[draw,fill=white]| 2 &
  |[draw,fill=light]| \textcolor{black}{6} &
  |[draw,fill=light]| \textcolor{black}{9} &
  |[draw,fill=white]| 0 &
  |[draw,fill=light]| \textcolor{black}{8} &
  |[draw,fill=white]| 1 &
  |[draw,fill=white]| 3 &
  |[draw,fill=light]| \textcolor{black}{4} &
  |[draw,fill=light]| \textcolor{black}{7} &
  |[draw,fill=light]| \textcolor{black}{5} \\
  |[draw,fill=white]| 3 &
  |[draw,fill=white]| 2 &
  |[draw,fill=light]| \textcolor{black}{5} &
  |[draw,fill=light]| \textcolor{black}{7} &
  |[draw,fill=white]| 0 &
  |[draw,fill=light]| \textcolor{black}{4} &
  |[draw,fill=white]| 1 &
  |[draw,fill=light]| \textcolor{black}{8} &
  |[draw,fill=light]| \textcolor{black}{9} &
  |[draw,fill=light]| \textcolor{black}{6} \\
  |[draw,fill=light]| \textcolor{black}{7} &
  |[draw,fill=light]| \textcolor{black}{9} &
  |[draw,fill=white]| 1 &
  |[draw,fill=light]| \textcolor{black}{8} &
  |[draw,fill=white]| 2 &
  |[draw,fill=white]| 3 &
  |[draw,fill=light]| \textcolor{black}{5} &
  |[draw,fill=white]| 0 &
  |[draw,fill=light]| \textcolor{black}{6} &
  |[draw,fill=light]| \textcolor{black}{4} \\
  |[draw,fill=dark]| \textcolor{white}{4} &
  |[draw,fill=white]| 0 &
  |[draw,fill=light]| \textcolor{black}{7} &
  |[draw,fill=light]| \textcolor{black}{5} &
  |[draw,fill=dark]| \textcolor{white}{6} &
  |[draw,fill=white]| 2 &
  |[draw,fill=light]| \textcolor{black}{9} &
  |[draw,fill=white]| 3 &
  |[draw,fill=white]| 1 &
  |[draw,fill=light]| \textcolor{black}{8} \\
  |[draw,fill=dark]| \textcolor{white}{5} &
  |[draw,fill=dark]| \textcolor{white}{4} &
  |[draw,fill=white]| 2 &
  |[draw,fill=white]| 3 &
  |[draw,fill=light]| \textcolor{black}{9} &
  |[draw,fill=light]| \textcolor{black}{8} &
  |[draw,fill=light]| \textcolor{black}{6} &
  |[draw,fill=white]| 1 &
  |[draw,fill=white]| 0 &
  |[draw,fill=light]| \textcolor{black}{7} \\
  |[draw,fill=light]| \textcolor{black}{6} &
  |[draw,fill=light]| \textcolor{black}{7} &
  |[draw,fill=dark]| \textcolor{white}{8} &
  |[draw,fill=dark]| \textcolor{white}{9} &
  |[draw,fill=white]| 1 &
  |[draw,fill=white]| 0 &
  |[draw,fill=light]| \textcolor{black}{4} &
  |[draw,fill=light]| \textcolor{black}{5} &
  |[draw,fill=white]| 3 &
  |[draw,fill=white]| 2 \\
  |[draw,fill=light]| \textcolor{black}{8} &
  |[draw,fill=white]| 3 &
  |[draw,fill=light]| \textcolor{black}{6} &
  |[draw,fill=white]| 1 &
  |[draw,fill=dark]| \textcolor{white}{4} &
  |[draw,fill=dark]| \textcolor{white}{5} &
  |[draw,fill=light]| \textcolor{black}{7} &
  |[draw,fill=light]| \textcolor{black}{9} &
  |[draw,fill=white]| 2 &
  |[draw,fill=white]| 0 \\
  |[draw,fill=light]| \textcolor{black}{9} &
  |[draw,fill=dark]| \textcolor{white}{8} &
  |[draw,fill=dark]| \textcolor{white}{4} &
  |[draw,fill=dark]| \textcolor{white}{6} &
  |[draw,fill=white]| 3 &
  |[draw,fill=dark]| \textcolor{white}{7} &
  |[draw,fill=white]| 0 &
  |[draw,fill=white]| 2 &
  |[draw,fill=light]| \textcolor{black}{5} &
  |[draw,fill=white]| 1 \\
  };\end{tikzpicture}
  \end{minipage}
  \hfill
  \begin{minipage}[c][1\width]{0.44\textwidth}
  \centering
  \begin{tikzpicture}
  \matrix (m) [matrix of nodes,nodes={element},column sep=-\pgflinewidth, row sep=-\pgflinewidth, label={[font=\large]above: type X}]{
  |[draw,fill=light]| \textcolor{black}{6} &
  |[draw,fill=white]| 1 &
  |[draw,fill=light]| \textcolor{black}{7} &
  |[draw,fill=white]| 2 &
  |[draw,fill=white]| 0 &
  |[draw,fill=light]| \textcolor{black}{8} &
  |[draw,fill=white]| 3 &
  |[draw,fill=light]| \textcolor{black}{9} &
  |[draw,fill=light]| \textcolor{black}{5} &
  |[draw,fill=light]| \textcolor{black}{4} \\
  |[draw,fill=light]| \textcolor{black}{7} &
  |[draw,fill=white]| 3 &
  |[draw,fill=white]| 0 &
  |[draw,fill=light]| \textcolor{black}{5} &
  |[draw,fill=light]| \textcolor{black}{9} &
  |[draw,fill=white]| 1 &
  |[draw,fill=light]| \textcolor{black}{4} &
  |[draw,fill=white]| 2 &
  |[draw,fill=light]| \textcolor{black}{8} &
  |[draw,fill=light]| \textcolor{black}{6} \\
  |[draw,fill=light]| \textcolor{black}{8} &
  |[draw,fill=light]| \textcolor{black}{7} &
  |[draw,fill=white]| 2 &
  |[draw,fill=white]| 0 &
  |[draw,fill=white]| 3 &
  |[draw,fill=light]| \textcolor{black}{4} &
  |[draw,fill=light]| \textcolor{black}{5} &
  |[draw,fill=light]| \textcolor{black}{6} &
  |[draw,fill=white]| 1 &
  |[draw,fill=light]| \textcolor{black}{9} \\
  |[draw,fill=light]| \textcolor{black}{9} &
  |[draw,fill=white]| 0 &
  |[draw,fill=light]| \textcolor{black}{5} &
  |[draw,fill=white]| 1 &
  |[draw,fill=light]| \textcolor{black}{8} &
  |[draw,fill=white]| 3 &
  |[draw,fill=light]| \textcolor{black}{6} &
  |[draw,fill=light]| \textcolor{black}{7} &
  |[draw,fill=light]| \textcolor{black}{4} &
  |[draw,fill=white]| 2 \\
  |[draw,fill=white]| 0 &
  |[draw,fill=white]| 2 &
  |[draw,fill=dark]| \textcolor{white}{4} &
  |[draw,fill=light]| \textcolor{black}{8} &
  |[draw,fill=dark]| \textcolor{white}{6} &
  |[draw,fill=light]| \textcolor{black}{9} &
  |[draw,fill=light]| \textcolor{black}{7} &
  |[draw,fill=white]| 1 &
  |[draw,fill=white]| 3 &
  |[draw,fill=light]| \textcolor{black}{5} \\
  |[draw,fill=white]| 1 &
  |[draw,fill=light]| \textcolor{black}{6} &
  |[draw,fill=dark]| \textcolor{white}{8} &
  |[draw,fill=light]| \textcolor{black}{4} &
  |[draw,fill=white]| 2 &
  |[draw,fill=dark]| \textcolor{white}{5} &
  |[draw,fill=white]| 0 &
  |[draw,fill=white]| 3 &
  |[draw,fill=light]| \textcolor{black}{9} &
  |[draw,fill=light]| \textcolor{black}{7} \\
  |[draw,fill=white]| 2 &
  |[draw,fill=dark]| \textcolor{white}{8} &
  |[draw,fill=white]| 1 &
  |[draw,fill=light]| \textcolor{black}{7} &
  |[draw,fill=dark]| \textcolor{white}{4} &
  |[draw,fill=light]| \textcolor{black}{6} &
  |[draw,fill=light]| \textcolor{black}{9} &
  |[draw,fill=light]| \textcolor{black}{5} &
  |[draw,fill=white]| 0 &
  |[draw,fill=white]| 3 \\
  |[draw,fill=white]| 3 &
  |[draw,fill=dark]| \textcolor{white}{4} &
  |[draw,fill=light]| \textcolor{black}{6} &
  |[draw,fill=dark]| \textcolor{white}{9} &
  |[draw,fill=light]| \textcolor{black}{5} &
  |[draw,fill=white]| 2 &
  |[draw,fill=light]| \textcolor{black}{8} &
  |[draw,fill=white]| 0 &
  |[draw,fill=light]| \textcolor{black}{7} &
  |[draw,fill=white]| 1 \\
  |[draw,fill=dark]| \textcolor{white}{4} &
  |[draw,fill=light]| \textcolor{black}{5} &
  |[draw,fill=light]| \textcolor{black}{9} &
  |[draw,fill=white]| 3 &
  |[draw,fill=white]| 1 &
  |[draw,fill=dark]| \textcolor{white}{7} &
  |[draw,fill=white]| 2 &
  |[draw,fill=light]| \textcolor{black}{8} &
  |[draw,fill=light]| \textcolor{black}{6} &
  |[draw,fill=white]| 0 \\
  |[draw,fill=dark]| \textcolor{white}{5} &
  |[draw,fill=light]| \textcolor{black}{9} &
  |[draw,fill=white]| 3 &
  |[draw,fill=dark]| \textcolor{white}{6} &
  |[draw,fill=light]| \textcolor{black}{7} &
  |[draw,fill=white]| 0 &
  |[draw,fill=white]| 1 &
  |[draw,fill=light]| \textcolor{black}{4} &
  |[draw,fill=white]| 2 &
  |[draw,fill=light]| \textcolor{black}{8} \\
  };\end{tikzpicture}
  \end{minipage}
\end{figure}
\begin{figure}[ht]
  \begin{minipage}[c][1\width]{0.44\textwidth}
  \centering
  \begin{tikzpicture}
  \matrix (m) [matrix of nodes,nodes={element},column sep=-\pgflinewidth, row sep=-\pgflinewidth, label={[font=\large]above: type X}]{
  |[draw,fill=white]| 0 &
  |[draw,fill=white]| 1 &
  |[draw,fill=white]| 3 &
  |[draw,fill=light]| \textcolor{black}{4} &
  |[draw,fill=light]| \textcolor{black}{5} &
  |[draw,fill=light]| \textcolor{black}{6} &
  |[draw,fill=white]| 2 &
  |[draw,fill=light]| \textcolor{black}{7} &
  |[draw,fill=light]| \textcolor{black}{8} &
  |[draw,fill=light]| \textcolor{black}{9} \\
  |[draw,fill=white]| 1 &
  |[draw,fill=light]| \textcolor{black}{7} &
  |[draw,fill=light]| \textcolor{black}{6} &
  |[draw,fill=white]| 3 &
  |[draw,fill=white]| 2 &
  |[draw,fill=light]| \textcolor{black}{8} &
  |[draw,fill=light]| \textcolor{black}{5} &
  |[draw,fill=light]| \textcolor{black}{9} &
  |[draw,fill=light]| \textcolor{black}{4} &
  |[draw,fill=white]| 0 \\
  |[draw,fill=white]| 2 &
  |[draw,fill=light]| \textcolor{black}{9} &
  |[draw,fill=white]| 1 &
  |[draw,fill=light]| \textcolor{black}{5} &
  |[draw,fill=light]| \textcolor{black}{6} &
  |[draw,fill=white]| 3 &
  |[draw,fill=light]| \textcolor{black}{7} &
  |[draw,fill=light]| \textcolor{black}{8} &
  |[draw,fill=white]| 0 &
  |[draw,fill=light]| \textcolor{black}{4} \\
  |[draw,fill=light]| \textcolor{black}{6} &
  |[draw,fill=white]| 2 &
  |[draw,fill=light]| \textcolor{black}{9} &
  |[draw,fill=light]| \textcolor{black}{8} &
  |[draw,fill=white]| 3 &
  |[draw,fill=white]| 1 &
  |[draw,fill=light]| \textcolor{black}{4} &
  |[draw,fill=white]| 0 &
  |[draw,fill=light]| \textcolor{black}{7} &
  |[draw,fill=light]| \textcolor{black}{5} \\
  |[draw,fill=white]| 3 &
  |[draw,fill=dark]| \textcolor{white}{6} &
  |[draw,fill=white]| 0 &
  |[draw,fill=light]| \textcolor{black}{7} &
  |[draw,fill=dark]| \textcolor{white}{8} &
  |[draw,fill=light]| \textcolor{black}{4} &
  |[draw,fill=light]| \textcolor{black}{9} &
  |[draw,fill=white]| 2 &
  |[draw,fill=light]| \textcolor{black}{5} &
  |[draw,fill=white]| 1 \\
  |[draw,fill=light]| \textcolor{black}{4} &
  |[draw,fill=dark]| \textcolor{white}{5} &
  |[draw,fill=white]| 2 &
  |[draw,fill=dark]| \textcolor{white}{9} &
  |[draw,fill=light]| \textcolor{black}{7} &
  |[draw,fill=white]| 0 &
  |[draw,fill=white]| 3 &
  |[draw,fill=white]| 1 &
  |[draw,fill=light]| \textcolor{black}{6} &
  |[draw,fill=light]| \textcolor{black}{8} \\
  |[draw,fill=light]| \textcolor{black}{5} &
  |[draw,fill=light]| \textcolor{black}{8} &
  |[draw,fill=dark]| \textcolor{white}{4} &
  |[draw,fill=white]| 0 &
  |[draw,fill=white]| 1 &
  |[draw,fill=dark]| \textcolor{white}{7} &
  |[draw,fill=light]| \textcolor{black}{6} &
  |[draw,fill=white]| 3 &
  |[draw,fill=light]| \textcolor{black}{9} &
  |[draw,fill=white]| 2 \\
  |[draw,fill=dark]| \textcolor{white}{7} &
  |[draw,fill=white]| 0 &
  |[draw,fill=light]| \textcolor{black}{5} &
  |[draw,fill=white]| 2 &
  |[draw,fill=light]| \textcolor{black}{4} &
  |[draw,fill=dark]| \textcolor{white}{9} &
  |[draw,fill=light]| \textcolor{black}{8} &
  |[draw,fill=light]| \textcolor{black}{6} &
  |[draw,fill=white]| 1 &
  |[draw,fill=white]| 3 \\
  |[draw,fill=light]| \textcolor{black}{8} &
  |[draw,fill=light]| \textcolor{black}{4} &
  |[draw,fill=dark]| \textcolor{white}{7} &
  |[draw,fill=white]| 1 &
  |[draw,fill=dark]| \textcolor{white}{9} &
  |[draw,fill=white]| 2 &
  |[draw,fill=white]| 0 &
  |[draw,fill=light]| \textcolor{black}{5} &
  |[draw,fill=white]| 3 &
  |[draw,fill=light]| \textcolor{black}{6} \\
  |[draw,fill=dark]| \textcolor{white}{9} &
  |[draw,fill=white]| 3 &
  |[draw,fill=light]| \textcolor{black}{8} &
  |[draw,fill=dark]| \textcolor{white}{6} &
  |[draw,fill=white]| 0 &
  |[draw,fill=light]| \textcolor{black}{5} &
  |[draw,fill=white]| 1 &
  |[draw,fill=light]| \textcolor{black}{4} &
  |[draw,fill=white]| 2 &
  |[draw,fill=light]| \textcolor{black}{7} \\
  };\end{tikzpicture}
  \end{minipage}
  \hfill
  \begin{minipage}[c][1\width]{0.44\textwidth}
  \centering
  \begin{tikzpicture}
  \matrix (m) [matrix of nodes,nodes={element},column sep=-\pgflinewidth, row sep=-\pgflinewidth, label={[font=\large]above: type X}]{
  |[draw,fill=white]| 0 &
  |[draw,fill=light]| \textcolor{black}{7} &
  |[draw,fill=white]| 2 &
  |[draw,fill=light]| \textcolor{black}{8} &
  |[draw,fill=white]| 1 &
  |[draw,fill=light]| \textcolor{black}{5} &
  |[draw,fill=light]| \textcolor{black}{9} &
  |[draw,fill=light]| \textcolor{black}{6} &
  |[draw,fill=white]| 3 &
  |[draw,fill=light]| \textcolor{black}{4} \\
  |[draw,fill=white]| 3 &
  |[draw,fill=white]| 2 &
  |[draw,fill=light]| \textcolor{black}{5} &
  |[draw,fill=light]| \textcolor{black}{4} &
  |[draw,fill=white]| 0 &
  |[draw,fill=light]| \textcolor{black}{8} &
  |[draw,fill=light]| \textcolor{black}{7} &
  |[draw,fill=white]| 1 &
  |[draw,fill=light]| \textcolor{black}{9} &
  |[draw,fill=light]| \textcolor{black}{6} \\
  |[draw,fill=light]| \textcolor{black}{6} &
  |[draw,fill=light]| \textcolor{black}{8} &
  |[draw,fill=white]| 0 &
  |[draw,fill=white]| 2 &
  |[draw,fill=light]| \textcolor{black}{7} &
  |[draw,fill=white]| 3 &
  |[draw,fill=white]| 1 &
  |[draw,fill=light]| \textcolor{black}{5} &
  |[draw,fill=light]| \textcolor{black}{4} &
  |[draw,fill=light]| \textcolor{black}{9} \\
  |[draw,fill=light]| \textcolor{black}{8} &
  |[draw,fill=light]| \textcolor{black}{9} &
  |[draw,fill=white]| 3 &
  |[draw,fill=white]| 0 &
  |[draw,fill=light]| \textcolor{black}{4} &
  |[draw,fill=white]| 1 &
  |[draw,fill=light]| \textcolor{black}{5} &
  |[draw,fill=white]| 2 &
  |[draw,fill=light]| \textcolor{black}{6} &
  |[draw,fill=light]| \textcolor{black}{7} \\
  |[draw,fill=white]| 1 &
  |[draw,fill=white]| 0 &
  |[draw,fill=dark]| \textcolor{white}{4} &
  |[draw,fill=light]| \textcolor{black}{7} &
  |[draw,fill=dark]| \textcolor{white}{9} &
  |[draw,fill=light]| \textcolor{black}{6} &
  |[draw,fill=white]| 3 &
  |[draw,fill=light]| \textcolor{black}{8} &
  |[draw,fill=white]| 2 &
  |[draw,fill=light]| \textcolor{black}{5} \\
  |[draw,fill=white]| 2 &
  |[draw,fill=dark]| \textcolor{white}{6} &
  |[draw,fill=light]| \textcolor{black}{9} &
  |[draw,fill=white]| 3 &
  |[draw,fill=light]| \textcolor{black}{5} &
  |[draw,fill=dark]| \textcolor{white}{7} &
  |[draw,fill=white]| 0 &
  |[draw,fill=light]| \textcolor{black}{4} &
  |[draw,fill=white]| 1 &
  |[draw,fill=light]| \textcolor{black}{8} \\
  |[draw,fill=light]| \textcolor{black}{4} &
  |[draw,fill=white]| 3 &
  |[draw,fill=dark]| \textcolor{white}{7} &
  |[draw,fill=light]| \textcolor{black}{5} &
  |[draw,fill=white]| 2 &
  |[draw,fill=dark]| \textcolor{white}{9} &
  |[draw,fill=light]| \textcolor{black}{6} &
  |[draw,fill=white]| 0 &
  |[draw,fill=light]| \textcolor{black}{8} &
  |[draw,fill=white]| 1 \\
  |[draw,fill=light]| \textcolor{black}{5} &
  |[draw,fill=light]| \textcolor{black}{4} &
  |[draw,fill=white]| 1 &
  |[draw,fill=dark]| \textcolor{white}{6} &
  |[draw,fill=dark]| \textcolor{white}{8} &
  |[draw,fill=white]| 0 &
  |[draw,fill=white]| 2 &
  |[draw,fill=light]| \textcolor{black}{9} &
  |[draw,fill=light]| \textcolor{black}{7} &
  |[draw,fill=white]| 3 \\
  |[draw,fill=dark]| \textcolor{white}{7} &
  |[draw,fill=white]| 1 &
  |[draw,fill=light]| \textcolor{black}{8} &
  |[draw,fill=dark]| \textcolor{white}{9} &
  |[draw,fill=light]| \textcolor{black}{6} &
  |[draw,fill=white]| 2 &
  |[draw,fill=light]| \textcolor{black}{4} &
  |[draw,fill=white]| 3 &
  |[draw,fill=light]| \textcolor{black}{5} &
  |[draw,fill=white]| 0 \\
  |[draw,fill=dark]| \textcolor{white}{9} &
  |[draw,fill=dark]| \textcolor{white}{5} &
  |[draw,fill=light]| \textcolor{black}{6} &
  |[draw,fill=white]| 1 &
  |[draw,fill=white]| 3 &
  |[draw,fill=light]| \textcolor{black}{4} &
  |[draw,fill=light]| \textcolor{black}{8} &
  |[draw,fill=light]| \textcolor{black}{7} &
  |[draw,fill=white]| 0 &
  |[draw,fill=white]| 2 \\
  };\end{tikzpicture}
  \end{minipage}
\end{figure}

\end{document}